\documentclass[10pt]{amsart}

\usepackage{amssymb,amsmath,amsthm}
\usepackage[dvips,ps,all]{xy}

\theoremstyle{plain}
\newtheorem{thm}[subsection]{Theorem}

\newtheorem{prop}[subsection]{Proposition}
\newtheorem{assumption}[subsection]{Basic Assumption}

\theoremstyle{definition}
\newtheorem{defn}[subsection]{Definition}

\theoremstyle{remark}
\newtheorem{rem}[subsection]{Remark}

\makeatletter 
\let\c@equation\c@subsection 
 
\makeatother

\newcommand{\ab}{{ \mathsf{ab} }}
\newcommand{\Ho}{{ \mathsf{Ho} }}

\newcommand{\sSet}{{ \mathsf{sSet} }}
\newcommand{\Chaincx}{{ \mathsf{Ch} }}
\newcommand{\sChaincx}{{ \mathsf{sCh} }}
\newcommand{\Mod}{{ \mathsf{Mod} }}
\newcommand{\sMod}{{ \mathsf{sMod} }}
\newcommand{\sSpectra}{{ \mathsf{sSp}^\Sigma }}
\newcommand{\Spectra}{{ \mathsf{Sp}^\Sigma }}

\newcommand{\C}{{ \mathsf{C} }}
\newcommand{\sC}{{ \mathsf{sC} }}
\newcommand{\D}{{ \mathsf{D} }}
\newcommand{\M}{{ \mathsf{M} }}
\newcommand{\sM}{{ \mathsf{sM} }}

\newcommand{\SymArray}{{ \mathsf{SymArray} }}
\newcommand{\SymSeq}{{ \mathsf{SymSeq} }}
\newcommand{\sSymArray}{{ \mathsf{sSymArray} }}
\newcommand{\sSymSeq}{{ \mathsf{sSymSeq} }}

\newcommand{\Set}{{ \mathsf{Set} }}

\newcommand{\Rt}{{ \mathsf{Rt} }}
\newcommand{\Lt}{{ \mathsf{Lt} }}
\newcommand{\Alg}{{ \mathsf{Alg} }}
\newcommand{\sRt}{{ \mathsf{sRt} }}
\newcommand{\sLt}{{ \mathsf{sLt} }}
\newcommand{\sAlg}{{ \mathsf{sAlg} }}
\newcommand{\LL}{{ \mathsf{L} }}

\newcommand{\unit}{{ \mathsf{k} }}

\newcommand{\AlgO}{{ \Alg_\capO }}
\newcommand{\LtO}{{ \Lt_\capO }}
\newcommand{\RtO}{{ \Rt_\capO }}
\newcommand{\sAlgO}{{ \sAlg_\capO }}
\newcommand{\sLtO}{{ \sLt_\capO }}
\newcommand{\sRtO}{{ \sRt_\capO }}

\newcommand{\capO}{{ \mathcal{O} }}

\newcommand{\realzn}{\mathrm{R}}
\newcommand{\NN}{{ \mathrm{N} }}

\newcommand{\Sk}{{ \mathrm{sk} }}

\newcommand{\Ev}{{ \mathrm{Ev} }}

\newcommand{\id}{{ \mathrm{id} }}
\newcommand{\op}{{ \mathrm{op} }}
\newcommand{\pr}{{ \mathrm{pr} }}

\newcommand{\Smash}{{ \,\wedge\, }}

\newcommand{\ol}[1]{{ \overline{{#1}} }}
\newcommand{\tensor}{{ \otimes }}

\newcommand{\tensorcheck}{{ \check{\tensor} }}
\newcommand{\tensordot}{{ \dot{\tensor} }}

\newcommand{\wequiv}{{ \ \simeq \ }}
\newcommand{\iso}{{ \cong }}
\newcommand{\Iso}{{  \ \cong \ }}

\newcommand{\rarrow}{{ \longrightarrow }}

\newcommand{\functor}[3]{{ {#1}\colon\thinspace{#2}\rarrow{#3} }}
\newcommand{\function}[3]{{ {#1}\colon\thinspace{#2}\rarrow{#3} }}

\newcommand{\subsetof}{{ \ \subset\ }}

\newdir{ >}{{}*!/-8pt/@{>}}
\newcommand{\monic}{ \ar@{ >->} }

\DeclareMathOperator*{\hocolim}{hocolim}

\DeclareMathOperator*{\colim}{colim}

\DeclareMathOperator{\Map}{Map}

\DeclareMathOperator{\BAR}{Bar}
\DeclareMathOperator{\Ab}{Ab}
\DeclareMathOperator{\Tot}{Tot}
\DeclareMathOperator{\U}{U}

\title[Bar constructions and Quillen homology]{Bar constructions and Quillen homology of modules over operads}

\author{John E. Harper}

\address{Institut de G\'eom\'etrie, Alg\`ebre et Topologie, EPFL, CH-1015 Lausanne, Switzerland}

\address{Department of Mathematics, University of Notre Dame, Notre Dame, IN 46556, USA}

\email{john.edward.harper@gmail.com}
\urladdr{http://sma.epfl.ch/~harper/}

\begin{document}
\maketitle

\section{Introduction}

There are many situations in algebraic topology, homotopy theory, and homological algebra in which operads parametrize interesting algebraic structures \cite{Fresse_lie_theory, Goerss_Hopkins_moduli_spaces, Kriz_May, Mandell, McClure_Smith_conjecture}. In many of these, there is a notion of abelianization or stabilization which provides a notion of homology \cite{Basterra, Basterra_Mandell, Goerss_f2_algebras, Schwede_cotangent, Schwede_algebraic}. In these contexts, Quillen's derived functor notion of homology is not just a graded collection of abelian groups, but a geometric object like a chain complex or spectrum, and distinct algebraic structures tend to have distinct notions of Quillen homology. For commutative algebras this is the cotangent complex appearing in Andr\'e-Quillen homology, and for the empty algebraic structure on spaces this is a chain complex calculating the singular homology of spaces. A useful introduction to Quillen homology is given in \cite{Goerss_Schemmerhorn}; see also the original articles \cite{Quillen, Quillen_rings}. In this paper we are interested in Quillen homology of algebras and modules over operads in symmetric spectra \cite{Hovey_Shipley_Smith} and unbounded chain complexes.  

Quillen homology provides very interesting invariants even in the case of simple algebraic structures such as commutative algebras; in \cite{Miller} Miller proves the Sullivan conjecture on maps from classifying spaces, and in his proof Quillen's derived functor notion of homology for commutative algebras is a critical ingredient.  This suggests that Quillen homology---for the larger class of algebraic structures parametrized by an action of an operad---will provide interesting and  useful invariants.

Consider any catgory $\C$ with all small limits, and with terminal object denoted by $*$. Let $\C_\ab$ denote the category of abelian group objects in $(\C,\times,*)$ and define \emph{abelianization} $\Ab$ to be the left adjoint 
\begin{align*}
\xymatrix{
  \C\ar@<0.5ex>[r]^-{\Ab} & \C_\ab\ar@<0.5ex>[l]^-{U}
}
\end{align*}
of the forgetful functor $U$, if it exists. Then if $\C$ and $\C_\ab$ are equipped with an appropriate homotopy theoretic structure, Quillen homology is the total left derived functor of abelianization; i.e., if $X\in\C$ then \emph{Quillen homology} of $X$ is by definition $\LL\Ab(X)$. This derived functor notion of homology is interesting in several contexts, including algebras and modules over augmented operads $\capO$ in unbounded chain complexes over a commutative ring $\unit$. In this context, the abelianization-forgetful adjunctions take the form of ``change of operads'' adjunctions
\begin{align*}
\xymatrix{
	\AlgO\ar@<0.5ex>[r]^-{I\circ_\capO(-)} & 
  \Alg_I=\Chaincx_\unit=(\AlgO)_\ab,\ar@<0.5ex>[l]
}\quad\quad
\xymatrix{
  \LtO\ar@<0.5ex>[r]^-{I\circ_\capO-} & 
  \Lt_{I}=\SymSeq=(\LtO)_\ab,\ar@<0.5ex>[l]
}
\end{align*}
with left adjoints on top, provided that $\capO[\mathbf{0}]=*$ and $\capO[\mathbf{1}]=\unit$; hence in this setting, abelianization is the ``indecomposables'' functor. Here, we denote by $\Chaincx_\unit$, $\AlgO$, $\SymSeq$, and $\LtO$ the categories of unbounded chain complexes over $\unit$, $\capO$-algebras, symmetric sequences, and left $\capO$-modules, respectively (Definitions \ref{defn:two_contexts}, \ref{defn:symmetric_sequences}, and \ref{defn:algebras_and_modules}).

When passing from the context of chain complexes to the context of symmetric spectra, abelian group objects appear less meaningful, and the interesting topological notion of homology is derived ``indecomposables''. If $X$ is an algebra or left module over an augmented operad $\capO$ in symmetric spectra, there are ``change of operads'' adjunctions
\begin{align*}
\xymatrix{
  \AlgO\ar@<0.5ex>[r]^-{I\circ_\capO(-)} & 
  \Alg_I=\Spectra,\ar@<0.5ex>[l]
}\quad\quad
\xymatrix{
  \LtO\ar@<0.5ex>[r]^-{I\circ_\capO-} & 
  \Lt_{I}=\SymSeq,\ar@<0.5ex>[l]
}
\end{align*}
with left adjoints on top.  If $\capO[\mathbf{0}]=*$ and $\capO[\mathbf{1}]=S$, then \emph{topological Quillen homology} (or Quillen homology) of $X$ is defined by $I\circ^\LL_\capO (X)$ for $\capO$-algebras and $I\circ^\LL_\capO X$ for left $\capO$-modules; hence in this setting, Quillen homology is the total left derived functor of ``indecomposables''. Here, we denote by $\Spectra$ the category of symmetric spectra (Definition \ref{defn:two_contexts}).

Using tools developed in the author's earlier homotopy theoretic work \cite{Harper_Modules, Harper_Spectra}, we show that the desired Quillen homology functors are well-defined and can be calculated as the realization of simplicial bar constructions (Definitions \ref{defn:realization}, \ref{defn:realzn_SymSeq}, and \ref{defn:simplicial_bar_constructions}), modulo cofibrancy conditions. The main theorem is this.

\begin{thm}\label{MainTheorem2}
Let $\capO$ be an augmented operad in symmetric spectra or unbounded chain complexes over $\unit$. Let $X$ be an $\capO$-algebra (resp. left $\capO$-module) and consider $\AlgO$ (resp. $\LtO$) with any of the model structures in Definition \ref{defn:stable_flat_positive_model_structures} or \ref{defn:model_structures_chain_complexes}. If the simplicial bar construction $\BAR(\capO,\capO,X)$ is objectwise cofibrant in $\AlgO$ (resp. $\LtO$), then there is a zig-zag of weak equivalences
\begin{align*}
  I\circ_\capO^\LL (X) &\wequiv
  |\BAR(I,\capO,X)|\\
  \Bigl(
  \text{resp.}\quad
  I\circ_\capO^\LL X &\wequiv
  |\BAR(I,\capO,X)|
  \Bigr)
\end{align*}
natural in such $X$; here, $\unit$ is any field of characteristic zero. In particular, in the context of symmetric spectra (resp. unbounded chain complexes over $\unit$), Quillen homology of $X$ is weakly equivalent to realization of the indicated simplicial bar construction, provided that $\capO[\mathbf{0}]=*$ and $\capO[\mathbf{1}]=S$ (resp. $\capO[\mathbf{0}]=*$ and $\capO[\mathbf{1}]=\unit$). 
\end{thm}

\begin{rem}
\label{rem:satisfying_the_conditions}
If $\capO$ is a cofibrant operad (Definition \ref{defn:cofibrant_operads}) and $X$ is a cofibrant $\capO$-algebra (resp. cofibrant left $\capO$-module), then by Proposition \ref{prop:cofibrant_operads} the simplicial bar construction $\BAR(\capO,\capO,X)$ is objectwise cofibrant in $\AlgO$ (resp. $\LtO$). More generally, if $\capO$ is an operad, $X$ is a cofibrant $\capO$-algebra (resp. cofibrant left $\capO$-module), and the forgetful functor from $\AlgO$ (resp. $\LtO$) to the underlying category preserves cofibrant objects, then the simplicial bar construction $\BAR(\capO,\capO,X)$ is objectwise cofibrant in $\AlgO$ (resp. $\LtO$); if $\unit$ is a field of characteristic zero, then this condition on the forgetful functor is satisfied for any operad $\capO$ in $\Chaincx_\unit$, since every object in $\Chaincx_\unit$ (resp. $\SymSeq$) is cofibrant. 
\end{rem}

\begin{rem}
If $\capO$ is an operad and $X$ is an $\capO$-algebra (resp. left $\capO$-module) such that $\capO[\mathbf{0}]=*$, $\capO$ is cofibrant in the underlying category $\SymSeq$, and $X$ is cofibrant in the underlying category, then the simplicial bar construction $\BAR(\capO,\capO,X)$ is objectwise cofibrant in $\AlgO$ (resp. $\LtO$).
\end{rem}

The condition in Theorem \ref{MainTheorem2}, that $\unit$ is a field of characteristic zero, ensures that the appropriate homotopy theoretic structures exist on the category of $\capO$-algebras and the category of left $\capO$-modules when $\capO$ is an arbitrary operad in unbounded chain complexes over $\unit$ \cite{Harper_Modules, Hinich}. The main results of this paper remain true when $\unit$ is a commutative ring, provided that the appropriate homotopy theoretic structures exist (Section \ref{sec:chain_complexes_ring}).

\subsection{Organization of the paper}

In Section \ref{sec:preliminaries} we recall some notation on algebras and modules over operads. In Section \ref{sec:model_structures} we recall certain model structures used in this paper and define homotopy colimits as total left derived functors of the colimit functors (Definition \ref{defn:hocolim}). In Section \ref{sec:hocolim_underlying} we warm-up with calculations of certain homotopy colimits in the underlying categories; the following is of particular interest.

\begin{thm}
\label{thm:hocolim_realzn}
If $X$ is a simplicial symmetric spectrum (resp. simplicial unbounded chain complex over $\unit$), then there is a zig-zag of weak equivalences
\begin{align*}
  \hocolim\limits_{\Delta^\op}X \wequiv |X|
\end{align*}
natural in $X$. Here, $\sSpectra$ (resp. $\sChaincx_\unit$) (Definition \ref{defn:diagram_categories}) is equipped with the projective model structure inherited from any of the monoidal model structures in Section \ref{sec:model_structures_symmetric_spectra} and $\unit$ is any commutative ring.
\end{thm}

Working with several model structures, we give a homotopical proof in Section \ref{sec:hocolim_calculations_algebraic} of the main theorem, once we have proved that certain homotopy colimits in $\capO$-algebras and left $\capO$-modules can be easily understood. The key result here, which is at the heart of this paper, is showing that the forgetful functor commutes with certain homotopy colimits. The theorem is this.

\begin{thm}\label{MainTheorem3}
Let $\capO$ be an operad in symmetric spectra or unbounded chain complexes over $\unit$. If $X$ is a simplicial $\capO$-algebra (resp. simplicial left $\capO$-module), then there are zig-zags of weak equivalences
\begin{align*}
  U\hocolim\limits^{\AlgO}_{\Delta^\op}X & \wequiv |U X| \wequiv
  \hocolim\limits_{\Delta^\op}U X \\
  \Bigl(
  \text{resp.}\quad
  U\hocolim\limits^{\LtO}_{\Delta^\op}X & \wequiv |U X| \wequiv
  \hocolim\limits_{\Delta^\op}U X
  \Bigr)
\end{align*}
natural in $X$. Here, $U$ is the forgetful functor, $\sAlgO$ (resp. $\sLtO$) (Definition \ref{defn:diagram_categories}) is equipped with the projective model structure inherited from any of the model structures in Definition \ref{defn:stable_flat_positive_model_structures} or  \ref{defn:model_structures_chain_complexes}, and $\unit$ is any field of characteristic zero.
\end{thm}

\begin{rem}
We sometimes decorate $\hocolim$ with $\AlgO$ or $\LtO$, as in Theorem \ref{MainTheorem3}, to emphasize these categories in the notation (Definition \ref{defn:hocolim}).
\end{rem}

A consequence of Theorem \ref{MainTheorem3} is that every $\capO$-algebra (resp. left $\capO$-module) is weakly equivalent to the homotopy colimit of its simplicial resolution. The theorem is this.

\begin{thm}\label{thm:fattened_replacement}
Let $\capO$ be an operad in symmetric spectra or unbounded chain complexes over $\unit$. If $X$ is an $\capO$-algebra (resp. left $\capO$-module), then there is a zig-zag of weak equivalences
\begin{align*}
  X & \wequiv \hocolim\limits^{\AlgO}_{\Delta^\op}\BAR(\capO,\capO,X)\\
  \Bigl(
  \text{resp.}\quad
  X &\wequiv \hocolim\limits^{\LtO}_{\Delta^\op}\BAR(\capO,\capO,X)
  \Bigr)
\end{align*}
in $\AlgO$ (resp. $\LtO$), natural in $X$. Here, $\sAlgO$ (resp. $\sLtO$) (Definition \ref{defn:diagram_categories}) is equipped with the projective model structure inherited from any of the model structures in Definition \ref{defn:stable_flat_positive_model_structures} or   \ref{defn:model_structures_chain_complexes}, and $\unit$ is any field of characteristic zero.
\end{thm}

Theorem \ref{thm:fattened_replacement} is a key result of this paper, and can be thought of as providing a particularly nice ``fattened'' replacement for $X$. The main theorem follows almost immediately; in fact, we prove a more general result on derived change of operads adjunctions. First we make the following observation. It turns out, we can use the techniques developed in \cite{Harper_Spectra}---in the context of symmetric spectra---to compare homotopy categories of algebras (resp. left modules) over operads in the context of unbounded chain complexes. The theorem is this.

\begin{thm}\label{MainTheorem4}
Let $\capO$ be an operad in symmetric spectra or unbounded chain complexes over $\unit$. If $\function{f}{\capO}{\capO'}$ is a map of operads, then the adjunction
\begin{align*}
\xymatrix{
  \Alg_{\capO}\ar@<0.5ex>[r]^-{f_*} & \Alg_{\capO'}\ar@<0.5ex>[l]^-{f^*}
}
\quad\quad
\Bigl(
\text{resp.}\quad
\xymatrix{
  \Lt_{\capO}\ar@<0.5ex>[r]^-{f_*} & \Lt_{\capO'}\ar@<0.5ex>[l]^-{f^*}
}
\Bigr)
\end{align*} 
is a Quillen adjunction with left adjoint on top and $f^*$ the forgetful functor. If furthermore, $f$ is an objectwise weak equivalence, then the adjunction is a Quillen equivalence, and hence induces an equivalence on the homotopy categories. Here, $\unit$ is any field of characteristic zero.
\end{thm}

\begin{proof}
The case for symmetric spectra is proved in \cite{Harper_Spectra}, and the case for unbounded chain complexes over $\unit$ is proved by the same argument.
\end{proof}

The main theorem is a particular case of the following more general result, which follows from Theorems \ref{MainTheorem3} and \ref{thm:fattened_replacement} together with the property that left Quillen functors commute with homotopy colimits (Proposition \ref{prop:commuting_with_hocolim}). The theorem is this.

\begin{thm}\label{thm:bar_calculates_derived_circle}
Let $\function{f}{\capO}{\capO'}$ be a morphism of operads in symmetric spectra or unbounded chain complexes over $\unit$. Let $X$ be an $\capO$-algebra (resp. left $\capO$-module) and consider $\AlgO$ (resp. $\LtO$) with any of the model structures in Definition \ref{defn:stable_flat_positive_model_structures} or \ref{defn:model_structures_chain_complexes}. If the simplicial bar construction $\BAR(\capO,\capO,X)$ is objectwise cofibrant in $\AlgO$ (resp. $\LtO$), then there is a zig-zag of weak equivalences
\begin{align*}
  \LL f_*(X)&\wequiv
  |\BAR(\capO',\capO,X)|
\end{align*}
in the underlying category, natural in such $X$. Here, $\LL f_*$ is the total left derived functor of $f_*$ and $\unit$ is any field of characteristic zero.
\end{thm}

In Sections \ref{sec:non_sigma_operads}, \ref{sec:chain_complexes_ring}, and \ref{sec:right_modules}, we indicate analogous results for the case of non-$\Sigma$ operads, operads in chain complexes over a commutative ring, and right modules over operads, respectively. Several proofs concerning homotopical analysis of the realization functors are deferred to Section \ref{sec:proofs}; the following is of particular interest.

\begin{prop}
Let $\function{f}{X}{Y}$ be a morphism of simplicial symmetric spectra (resp. simplicial unbounded chain complexes over $\unit$). If $f$ is an objectwise weak equivalence, then $\function{|f|}{|X|}{|Y|}$ is a weak equivalence. Here, $\unit$ is any commutative ring.
\end{prop}

In Section \ref{sec:cofibrant_operads} we prove that the forgetful functor from $\capO$-algebras (resp. left $\capO$-modules) to the underlying category preserves cofibrant objects, provided that $\capO$ is a cofibrant operad; this is used in Remark \ref{rem:satisfying_the_conditions}.

\subsection{Relationship to previous work}

One of the results of Basterra \cite{Basterra} is that in the context of $S$-modules \cite{EKMM}, and for non-unital commutative $S$-algebras, the total left derived ``indecomposables'' functor is well-defined and can be calculated as realization of a simplicial bar construction. Theorem~\ref{MainTheorem2} improves this result to algebras and left modules over any augmented operad in symmetric spectra, and also provides a simplified homotopical proof---in the context of symmetric spectra---of Basterra's original result.

One of the theorems of Fresse~\cite{Fresse} is that in the context of non-negative chain complexes over a field of characteristic zero, and for left modules and augmented operads which are trivial at zero---such modules do not include algebras over operads---then under additional conditions, the total left derived ``indecomposables'' functor is well-defined and can be calculated as realization of a simplicial bar construction. Theorem~\ref{MainTheorem2} improves this result to include algebras over augmented operads. Theorem \ref{MainTheorem2} also improves this result to the context of unbounded chain complexes over a field of characteristic zero, and to left modules over augmented operads (not necessarily trivial at zero), and also provides a simplified homotopical proof of Fresse's original result.

One of the theorems of Hinich~\cite{Hinich} is that for unbounded chain complexes over a field of characteristic zero, a morphism of operads which is an objectwise weak equivalence induces a Quillen equivalence between categories of algebras over operads. Theorem~\ref{MainTheorem4} improves this result to the category of left modules over operads.

\subsection*{Acknowledgments}

The author would like to thank Bill Dwyer for his constant encouragement and invaluable help and advice, and Benoit Fresse and Charles Rezk for helpful comments. The research for part of this paper was carried out, while the author was a visiting researcher at the Thematic Program on Geometric Applications of Homotopy Theory at the Fields Institute, Toronto.

\section{Preliminaries on algebras and modules over operads}
\label{sec:preliminaries}

The purpose of this section is to recall various definitions and constructions associated to symmetric sequences and algebras and modules over operads. In this paper, we work in the following two contexts.

\begin{defn}\ 
\label{defn:two_contexts}
\begin{itemize}
\item Denote by $(\Spectra,\Smash,S)$ the closed symmetric monoidal category of symmetric spectra \cite{Hovey_Shipley_Smith, Schwede_book_project}.
\item Denote by $(\Chaincx_\unit,\tensor,\unit)$ the closed symmetric monoidal category of unbounded chain complexes over $\unit$ \cite{Hovey, MacLane_homology}. 
\end{itemize}
Here, $\unit$ is any commutative ring. Both categories have all small limits and colimits; the null object is denoted by $*$. 

\end{defn}

\begin{rem}
By \emph{closed} we mean there exists a functor 
\begin{align*}
  \bigl(\Spectra\bigr)^\op\times\Spectra\rarrow \Spectra, 
  \quad\quad &(Y,Z)\longmapsto \Map(Y,Z),\\
  \Bigl(
  \text{resp.}\quad
  \Chaincx_\unit^\op\times\Chaincx_\unit\rarrow \Chaincx_\unit, 
  \quad\quad &(Y,Z)\longmapsto \Map(Y,Z),
  \Bigr)
\end{align*}
which we call \emph{mapping object}, which fits into isomorphisms 
\begin{align*}
   \hom(X\Smash Y,Z)&\Iso \hom(X,\Map(Y,Z))\\
   \Bigl(
   \text{resp.}\quad
   \hom(X\tensor Y,Z)&\Iso \hom(X,\Map(Y,Z))
   \Bigr)
\end{align*}
natural in $X,Y,Z$.
\end{rem}

\subsection{Symmetric sequences, tensor products, and circle products}

The purpose of this section is to recall certain details---of two monoidal structures on symmetric sequences---which will be needed in this paper. A fuller account of the material in this section is given in \cite{Harper_Modules}, which was largely influenced by the development in \cite{Rezk}; see also \cite{Fresse_lie_theory, Fresse_modules, Kapranov_Manin}. 

Define the sets $\mathbf{n}:=\{1,\dots,n\}$ for each $n\geq 0$, where $\mathbf{0}:=\emptyset$ denotes the empty set. If $T$ is a finite set, define $|T|$ to be the number of elements in $T$.

\begin{defn}
\label{defn:symmetric_sequences}
Let $n\geq 0$.
\begin{itemize}
\item $\Sigma$ is the category of finite sets and their bijections. 
\item A \emph{symmetric sequence} in $\Spectra$ (resp. $\Chaincx_\unit$) is a functor $\functor{A}{\Sigma^{\op}}{\Spectra}$ (resp. $\functor{A}{\Sigma^{\op}}{\Chaincx_\unit}$). Denote by $\SymSeq$ the category of symmetric sequences in $\Spectra$ (resp. $\Chaincx_\unit$) and their natural transformations. 
\item A symmetric sequence $A$ is \emph{concentrated at $n$} if $A[\mathbf{r}]=*$ for all $r\neq n$.
\end{itemize}
\end{defn}

To remain consistent with \cite{Harper_Spectra}, and to avoid confusion with other tensor products appearing in this paper, we use the following $\tensorcheck$ notation.

\begin{defn}
Consider symmetric sequences in $\Spectra$ (resp. in $\Chaincx_\unit$).
Let $A_1,\dotsc,A_t\in\SymSeq$. The \emph{tensor products} $A_1\tensorcheck\dotsb\tensorcheck A_t\in\SymSeq$ are the left Kan extensions of objectwise smash (resp. objectwise tensor) along coproduct of sets
\begin{align*}
\xymatrix{
  (\Sigma^{\op})^{\times t}
  \ar[rr]^-{A_1\times\dotsb\times A_t}\ar[d]^{\coprod} & &
  \bigl(\Spectra\bigr)^{\times t}\ar[r]^-{\Smash} & \Spectra \\
  \Sigma^{\op}\ar[rrr]^{A_1\tensorcheck\dotsb\tensorcheck
  A_t}_{\text{left Kan extension}} & & & \Spectra
}\quad
\xymatrix{
  (\Sigma^{\op})^{\times t}
  \ar[rr]^-{A_1\times\dotsb\times A_t}\ar[d]^{\coprod} & &
  \bigl(\Chaincx_\unit\bigr)^{\times t}\ar[r]^-{\tensor} & \Chaincx_\unit \\
  \Sigma^{\op}\ar[rrr]^{A_1\tensorcheck\dotsb\tensorcheck
  A_t}_{\text{left Kan extension}} & & & \Chaincx_\unit
}
\end{align*}
\end{defn}

The following calculations will be useful when working with tensor products.

\begin{prop}
Consider symmetric sequences in $\Spectra$ (resp. in $\Chaincx_\unit$). Let $A_1,\dotsc,A_t\in\SymSeq$ and $R\in\Sigma$, with $r:=|R|$. There are natural isomorphisms
\begin{align*}
  (A_1\tensorcheck\dotsb\tensorcheck A_t)[R]&\Iso\ 
  \coprod_{\substack{\function{\pi}{R}{\mathbf{t}}\\ \text{in $\Set$}}}
  A_1[\pi^{-1}(1)]\Smash\dotsb\Smash
  A_t[\pi^{-1}(t)],\\
  &\Iso
  \coprod_{r_1+\dotsb +r_t=r}A_1[\mathbf{r_1}]\Smash\dotsb\Smash 
  A_t[\mathbf{r_t}]\underset{{\Sigma_{r_1}\times\dotsb\times
  \Sigma_{r_t}}}{\cdot}\Sigma_{r}\\
  \text{resp.}\quad
  (A_1\tensorcheck\dotsb\tensorcheck A_t)[R]&\Iso\ 
  \coprod_{\substack{\function{\pi}{R}{\mathbf{t}}\\ \text{in $\Set$}}}
  A_1[\pi^{-1}(1)]\tensor\dotsb\tensor
  A_t[\pi^{-1}(t)],\\
  &\Iso
  \coprod_{r_1+\dotsb +r_t=r}A_1[\mathbf{r_1}]\tensor\dotsb\tensor 
  A_t[\mathbf{r_t}]\underset{{\Sigma_{r_1}\times\dotsb\times
  \Sigma_{r_t}}}{\cdot}\Sigma_{r}.
\end{align*}
\end{prop}

Here, $\Set$ is the category of sets and their maps. It will be conceptually useful to extend the definition of tensor powers $A^{\tensorcheck t}$ to situations in which the integers $t$ are replaced by a finite set $T$.

\begin{defn}
Consider symmetric sequences in $\Spectra$ (resp. in $\Chaincx_\unit$).
Let $A\in\SymSeq$ and $R,T\in\Sigma$. The \emph{tensor powers} $A^{\tensorcheck T}\in\SymSeq$ are defined objectwise by
\begin{align*}
  (A^{\tensorcheck\emptyset})[R]:= 
  \coprod_{\substack{\function{\pi}{R}{\emptyset}\\ \text{in $\Set$}}}
  S,\quad\quad
  &(A^{\tensorcheck T})[R]:= 
  \coprod_{\substack{\function{\pi}{R}{T}\\ \text{in $\Set$}}}
  \Smash_{t\in T} A[\pi^{-1}(t)]\quad(T\neq\emptyset)\\
  \Bigl(
  \text{resp.}\quad
  (A^{\tensorcheck\emptyset})[R]:= 
  \coprod_{\substack{\function{\pi}{R}{\emptyset}\\ \text{in $\Set$}}}
  \unit,\quad\quad
  &(A^{\tensorcheck T})[R]:= 
  \coprod_{\substack{\function{\pi}{R}{T}\\ \text{in $\Set$}}}
  \tensor_{t\in T} A[\pi^{-1}(t)]\quad(T\neq\emptyset)
  \Bigr).
\end{align*}
Note that there are no functions $\function{\pi}{R}{\emptyset}$ in $\Set$ unless $R=\emptyset$. We will use the abbreviation $A^{\tensorcheck 0}:=A^{\tensorcheck\emptyset}$.  
\end{defn}

\begin{defn}\label{defn:circle_product}
Consider symmetric sequences in $\Spectra$ (resp. in $\Chaincx_\unit$). Let $A,B,C\in\SymSeq$, and $r,t\geq 0$. The \emph{circle product} (or composition product) $A\circ B\in\SymSeq$ is defined objectwise by the coend
\begin{align*}
  (A\circ B)[\mathbf{r}] := A\Smash_\Sigma (B^{\tensorcheck-})[\mathbf{r}]
  &\Iso 
  \coprod_{t\geq 0}A[\mathbf{t}]\Smash_{\Sigma_t}
  (B^{\tensorcheck t})[\mathbf{r}]\\
  \Bigl(
  \text{resp.}\quad
  (A\circ B)[\mathbf{r}] := A\tensor_\Sigma (B^{\tensorcheck-})[\mathbf{r}]
  &\Iso 
  \coprod_{t\geq 0}A[\mathbf{t}]\tensor_{\Sigma_t}
  (B^{\tensorcheck t})[\mathbf{r}]
  \Bigr),
\end{align*}
and the \emph{mapping sequence} $\Map^\circ(B,C)\in\SymSeq$ is defined objectwise by the end
\begin{align*}
  \Map^\circ(B,C)[\mathbf{t}] := \Map((B^{\tensorcheck \mathbf{t}})[-],C)^\Sigma \Iso 
  \prod_{r\geq 0}\Map((B^{\tensorcheck \mathbf{t}})[\mathbf{r}],
  C[\mathbf{r}])^{\Sigma_r}.
\end{align*}
\end{defn}

These mapping sequences---which arise explicitly in Sections  \ref{sec:modules_algebras} and \ref{sec:cofibrant_operads}---are part of a closed monoidal category structure on symmetric sequences and fit into isomorphisms
\begin{align}
\label{eq:circle_mapping_sequence_adjunction}
  \hom(A\circ B,C)&\Iso\hom(A,\Map^\circ(B,C))
\end{align}
natural in symmetric sequences $A,B,C$.

\begin{prop}
Consider symmetric sequences in $\Spectra$ (resp. in $\Chaincx_\unit$). 
\begin{itemize}
\item [(a)] $(\SymSeq,\tensorcheck,1)$ has the structure of a closed symmetric monoidal category with all small limits and colimits. The unit for $\tensorcheck$ denoted ``$1$'' is the symmetric sequence concentrated at $0$ with value $S$ (resp. $\unit$).
\item [(b)] $(\SymSeq,\circ,I)$ has the structure of a closed monoidal category with all small limits and colimits. The unit for $\circ$ denoted ``$I$'' is the symmetric sequence concentrated at $1$ with value $S$ (resp. $\unit$). Circle product is not symmetric.
\end{itemize}
\end{prop}

\begin{defn}
Let $Z$ be a symmetric spectrum (resp. unbounded chain complex over $\unit$). Define $\hat{Z}\in\SymSeq$ to be the symmetric sequence concentrated at $0$ with value $Z$.
\end{defn}

The category $\Spectra$ (resp. $\Chaincx_\unit$) embeds in $\SymSeq$ as the full subcategory of symmetric sequences concentrated at $0$, via the functor 
\begin{align*}
  \Spectra\rarrow\SymSeq,\quad\quad &Z\longmapsto \hat{Z}\\
  \Bigl(
  \text{resp.}\quad
  \Chaincx_\unit\rarrow\SymSeq,\quad\quad &Z\longmapsto \hat{Z}
  \Bigr).
\end{align*}

\begin{defn}\label{defn:corresponding_functor}
Consider symmetric sequences in $\Spectra$ (resp. in $\Chaincx_\unit$). Let $\capO$ be a symmetric sequence and $Z\in\Spectra$ (resp. $Z\in\Chaincx_\unit$). The corresponding functor $\functor{\capO}{\Spectra}{\Spectra}$ (resp. $\functor{\capO}{\Chaincx_\unit}{\Chaincx_\unit}$) is defined objectwise by,
\begin{align*}
  \capO(Z)&:=\capO\circ(Z):=\coprod\limits_{t\geq 0}\capO[\mathbf{t}]
  \Smash_{\Sigma_t}Z^{\wedge t}\Iso (\capO\circ\hat{Z})[\mathbf{0}]\\
  \Bigl(
  \text{resp.}\quad
  \capO(Z)&:=\capO\circ(Z):=\coprod\limits_{t\geq 0}\capO[\mathbf{t}]
  \tensor_{\Sigma_t}Z^{\tensor t}\Iso (\capO\circ\hat{Z})[\mathbf{0}]
  \Bigr).
\end{align*}
\end{defn}

\subsection{Algebras and modules over operads}
\label{sec:modules_algebras}

The purpose of this section is to recall certain definitions and properties of algebras and modules over operads that will be needed in this paper. A useful introduction to operads and their algebras is given in \cite{Kriz_May}; see also the original article \cite{May}.

\begin{defn}
An \emph{operad} is a monoid object in $(\SymSeq,\circ,I)$ and a \emph{morphism of operads} is a morphism of monoid objects in $(\SymSeq,\circ,I)$.
\end{defn}

An introduction to monoid objects and monoidal categories is given in \cite[VII]{MacLane_categories}. Each operad $\capO$ in symmetric spectra (resp. unbounded chain complexes over $\unit$) determines a functor $\function{\capO}{\Spectra}{\Spectra}$  (resp. $\function{\capO}{\Chaincx_\unit}{\Chaincx_\unit}$) (Definition \ref{defn:corresponding_functor}) together with natural transformations $\function{m}{\capO\capO}{\capO}$ and $\function{\eta}{\id}{\capO}$ which give the functor $\function{\capO}{\Spectra}{\Spectra}$ (resp. $\function{\capO}{\Chaincx_\unit}{\Chaincx_\unit}$) the structure of a monad (or triple). One perspective offered in \cite[I.2 and I.3]{Kriz_May} is that operads determine particularly manageable monads. For a useful introduction to monads and their algebras, see \cite[VI]{MacLane_categories}.

\begin{defn}
\label{defn:algebras_and_modules}
Let $\capO$ be an operad in symmetric spectra (resp. unbounded chain complexes over $\unit$).
\begin{itemize}
\item A \emph{left $\capO$-module} is an object in $(\SymSeq,\circ,I)$ with a left action of $\capO$ and a \emph{morphism of left $\capO$-modules} is a map which respects the left $\capO$-module structure. Denote by $\LtO$ the category of left $\capO$-modules and their morphisms.
\item A \emph{right $\capO$-module} is an object in $(\SymSeq,\circ,I)$ with a right action of $\capO$ and a \emph{morphism of right $\capO$-modules} is a map which respects the right $\capO$-module structure.  Denote by $\RtO$ the category of right $\capO$-modules and their morphisms.
\item An \emph{$\capO$-algebra} is an object in $\Spectra$ (resp. $\Chaincx_\unit$) with a left action of the monad $\functor{\capO}{\Spectra}{\Spectra}$ (resp. $\functor{\capO}{\Chaincx_\unit}{\Chaincx_\unit}$) and a \emph{morphism of $\capO$-algebras} is a map in $\Spectra$ (resp. $\Chaincx_\unit$) which respects the left action of the monad. Denote by $\AlgO$ the category of $\capO$-algebras and their morphisms.
\end{itemize}
\end{defn}

It is easy to verify that an $\capO$-algebra is the same as an object $X$ in $\Spectra$ (resp. $\Chaincx_\unit$) with a left $\capO$-module structure on $\hat{X}$, and if $X$ and $X'$ are $\capO$-algebras, then a morphism of $\capO$-algebras is the same as a map $\function{f}{X}{X'}$ in $\Spectra$ (resp. $\Chaincx_\unit$) such that $\function{\hat{f}}{\hat{X}}{\hat{X'}}$ is a morphism of left $\capO$-modules. In other words, an algebra over an operad $\capO$ is the same as a left $\capO$-module which is concentrated at $0$, and $\AlgO$ embeds in $\LtO$ as the full subcategory of left $\capO$-modules concentrated at $0$, via the functor 
\begin{align*}
  \AlgO\rarrow\LtO,\quad\quad Z\longmapsto \hat{Z}.
\end{align*}
It follows easily from \eqref{eq:circle_mapping_sequence_adjunction} that giving a symmetric sequence $Y$ a left $\capO$-module structure is the same as giving a morphism of operads 
\begin{align}\label{eq:operad_action}
  \function{m}{\capO}{\Map^\circ(Y,Y)}.
\end{align}
Similarly, giving an object $X$ in $\Spectra$ (resp. in $\Chaincx_\unit$) an $\capO$-algebra structure is the same as giving a morphism of operads
\begin{align*}
  \function{m}{\capO}{\Map^\circ(\hat{X},\hat{X})}.
\end{align*}
This is the original definition given in \cite{May} of an $\capO$-algebra structure on $X$, where $\Map^\circ(\hat{X},\hat{X})$ is called the \emph{endomorphism operad} of $X$. These correspondences will be particularly useful in the Section \ref{sec:cofibrant_operads}.

\begin{prop}\label{prop:basic_properties_LtO}
Let $\capO$ be an operad in symmetric spectra (resp. unbounded chain complexes over $\unit$).
\begin{itemize}
\item[(a)] There are adjunctions
\begin{align*}
\xymatrix{
  \SymSeq\ar@<0.5ex>[r]^-{\capO\circ-} & \LtO,\ar@<0.5ex>[l]^-{U}
}\quad\quad
\xymatrix{
  \Spectra\ar@<0.5ex>[r]^-{\capO\circ(-)} & \AlgO,\ar@<0.5ex>[l]^-{U}
}
\quad\quad
\Bigl(
\text{resp.}\quad
\xymatrix{
  \Chaincx_\unit\ar@<0.5ex>[r]^-{\capO\circ(-)} & 
  \AlgO,\ar@<0.5ex>[l]^-{U}
}
\Bigr)
\end{align*}
with left adjoints on top and $U$ the forgetful functor.
\item[(b)] All small colimits exist in $\LtO$ and $\AlgO$, and both reflexive coequalizers and filtered colimits are preserved by the forgetful functors.
\item[(c)] All small limits exist in $\LtO$ and $\AlgO$, and are preserved by the forgetful functors.
\end{itemize}
\end{prop}

We will recall the definition of reflexive coequalizers in Section \ref{sec:reflexive_coequalizers}.

\section{Model structures and homotopy colimits}
\label{sec:model_structures}

The purpose of this section is to define homotopy colimits as total left derived functors of the colimit functors. Model categories provide a setting in which one can do homotopy theory, and in particular, provide a framework for constructing and calculating such derived functors. A useful introduction to model categories is given in \cite{Dwyer_Spalinski}; see also \cite{Chacholski_Scherer, Goerss_Jardine, Hirschhorn, Hovey} and the original articles \cite{Quillen, Quillen_rational}. The extra structure of a cofibrantly generated model category is described in \cite[2.2]{Schwede_Shipley}; for further discussion see \cite[Chapter 11]{Hirschhorn} and \cite[2.1]{Hovey}. When we refer to the extra structure of a monoidal model category, we are using \cite[3.1]{Schwede_Shipley}. 

\subsection{Model structures}
\label{sec:model_structures_symmetric_spectra}

The purpose of this section is to recall certain model category structures that will be needed in this paper. In the case of symmetric spectra, we use several different model structures, each of which has the same weak equivalences. 

The \emph{stable model structure} on $\Spectra$, which has weak equivalences the stable equivalences and fibrations the stable fibrations, is one of several model category structures that is proved in \cite{Hovey_Shipley_Smith} to exist on symmetric spectra. When working with commutative ring spectra, or more generally, algebras over operads in spectra, the following positive variant of the stable model structure is useful. The \emph{positive stable model structure} on $\Spectra$, which has weak equivalences the stable equivalences and fibrations the positive stable fibrations, is proved in \cite{Mandell_May_Schwede_Shipley} to exist on symmetric spectra. It is often useful to work with the following flat variant of the (positive) stable model structure, since the flat variant has more cofibrations. The \emph{(positive) flat stable model structure} on $\Spectra$, which has weak equivalences the stable equivalences and fibrations the (positive) flat stable fibrations, is proved in \cite{Shipley_comm_ring} to exist on symmetric spectra. In addition to the references cited above, see also \cite[Section 4]{Harper_Spectra} for a description of the cofibrations in each model structure on symmetric spectra described above.

\begin{rem}
\label{rem:flat_notation}
For ease of notational purposes, we have followed Schwede \cite{Schwede_book_project} in using the term \emph{flat} (e.g., flat stable model structure) for what is called $S$ (e.g., stable $S$-model structure) in \cite{Hovey_Shipley_Smith, Schwede, Shipley_comm_ring}. For some of the good properties of the flat stable model structure, see \cite[5.3.7 and 5.3.10]{Hovey_Shipley_Smith}.
\end{rem}

Each model structure on symmetric spectra described above is cofibrantly generated in which the generating cofibrations and acyclic cofibrations have small domains, and that with respect to each model structure $(\Spectra,\Smash,S)$ is a monoidal model category. It is easy to check that the diagram category $\SymSeq$ inherits corresponding projective model category structures, where the weak equivalences (resp. fibrations) are the objectwise weak equivalences (resp. objectwise fibrations). We refer to these model structures by the names above (e.g., the \emph{positive flat stable} model structure on $\SymSeq$). Each of these model structures is cofibrantly generated in which the generating cofibrations and acyclic cofibrations have small domains. Furthermore, with respect to each model structure $(\SymSeq,\tensor,1)$ is a monoidal model category. 

In this paper we will often use implicitly a model structure on $\Spectra$, called the \emph{injective stable model structure} in \cite{Hovey_Shipley_Smith}, which has weak equivalences the stable equivalences and cofibrations the monomorphisms \cite[5.3]{Hovey_Shipley_Smith}; for instance, in the proof of Proposition \ref{prop:comparing_hocolim_and_realzn} and other similar arguments. The injective stable model structure on symmetric spectra is useful---it has more cofibrations than the flat stable model structure---but it is not a monoidal model structure on $(\Spectra,\Smash,S)$.

One of the advantages of the positive (flat) stable model structures on symmetric spectra described above is that they induce corresponding model category structures on algebras (resp. left modules) over an operad. It is proved in \cite{Harper_Spectra} that the following model category structures exist.

\begin{defn}
\label{defn:stable_flat_positive_model_structures}
Let $\capO$ be an operad in symmetric spectra. 
\begin{itemize}
\item[(a)] The \emph{positive flat stable model structure} on $\AlgO$ (resp. $\LtO$) has weak equivalences the stable equivalences (resp. objectwise stable equivalences) and fibrations the positive flat stable fibrations (resp. objectwise positive flat stable fibrations).
\item[(b)] The \emph{positive stable model structure} on $\AlgO$ (resp. $\LtO$) has weak equivalences the stable equivalences (resp. objectwise stable equivalences) and fibrations the positive stable fibrations (resp. objectwise positive stable fibrations).
\end{itemize}
\end{defn}

\begin{rem}
In this paper, we give our proofs for the (positive) flat stable model structure when working in the context of symmetric spectra. Our results remain true for the (positive) stable model structure, since it is easily checked that every (positive) stable cofibration is a (positive) flat stable cofibration.
\end{rem}

\begin{defn}
\label{defn:model_structures_chain_complexes}
Let $\capO$ be an operad in unbounded chain complexes over $\unit$. 
\begin{itemize}
\item[(a)] The model structure on $\Chaincx_\unit$ (resp. $\SymSeq$) has weak equivalences the homology isomorphisms (resp. objectwise homology isomorphisms) and fibrations the dimensionwise surjections (resp. objectwise dimensionwise surjections); here, $\unit$ is any commutative ring.
\item[(b)] The model structure on $\AlgO$ (resp. $\LtO$) has weak equivalences the homology isomorphisms (resp. objectwise homology isomorphisms) and fibrations the dimensionwise surjections (resp. objectwise dimensionwise surjections); here, $\unit$ is any field of characteristic zero.
\end{itemize}
\end{defn}

It is proved in \cite{Hovey} that the model structure described in Definition \ref{defn:model_structures_chain_complexes}(a) exists on unbounded chain complexes over $\unit$, and it is easy to check that the diagram category $\SymSeq$ inherits the corresponding projective model category structure. The model structures described in Definition \ref{defn:model_structures_chain_complexes}(b) are proved to exist in \cite{Harper_Modules}; for the case of $\capO$-algebras, see also the earlier paper \cite{Hinich} which uses different arguments.

\subsection{Homotopy colimits and simplicial objects}
The purpose of this section is to define homotopy colimits of simplicial objects. Useful introductions to simplicial sets are given in \cite{Dwyer_Henn, Gabriel_Zisman, Goerss_Jardine, May_simplicial}; see also the presentations in \cite{Goerss_Schemmerhorn, Hovey, Weibel}. 

Define the totally ordered sets $[n]:=\{0,1,\dotsc,n\}$ for each $n\geq 0$, and given their natural ordering.

\begin{defn} 
\label{defn:diagram_categories}
Let $\M$ be a category with all small limits and colimits.
\begin{itemize}
\item $\Delta$ is the category with objects the totally ordered sets $[n]$ for $n\geq 0$ and morphisms the maps of sets $\function{\xi}{[n]}{[n']}$ which respect the ordering; i.e., such that $k\leq l$ implies $\xi(k)\leq\xi(l)$.
\item A \emph{simplicial object} in $\M$ is a functor $\functor{X}{\Delta^{\op}}{\M}$. Denote by $\sM:=\M^{\Delta^\op}$ the category of simplicial objects in $\M$ and their natural transformations.
\item If $X\in\sM$, we will sometimes use the notation
$\pi_0 X := \colim\bigl(\functor{X}{\Delta^\op}{\M}\bigr)$.
\item If $X\in\sM$ and $n\geq 0$, we usually use the notation $X_n:=X([n])$.
\item If $\D$ is a small category and $\function{X}{\D}{\M}$ is a functor, we will sometimes use the notation 
\begin{align*}
  \colim^\M_\D X:=\colim_\D X
\end{align*} 
to emphasize the target category $\M$ of the colimit functor $\M^\D\rarrow\M$.
\item $\emptyset$ denotes an initial object in $\M$ and $*$ denotes a terminal object in $\M$.
\item For each $n\geq 0$, the \emph{standard $n$-simplex} $\Delta[n]$ is the simplicial set with $k$-simplices the morphisms in $\Delta$ from $[k]$ to $[n]$; i.e.,
$\Delta[n]_k:=\hom_\Delta([k],[n])$.
\end{itemize}
\end{defn}

In particular, we denote by $\sSet$ the category of simplicial sets and by $\sSet_*$ the category of pointed simplicial sets.

\begin{defn}
Let $\M$ be a category with all small colimits. If $X\in\sM$ (resp. $X\in\M$) and $K\in\sSet$, then $X\cdot K\in\sM$ is defined objectwise by
\begin{align*}
  (X\cdot K)_n:=\coprod\limits_{K_n}X_n
  \quad\quad
  \Bigl(\text{resp.}\quad
  (X\cdot K)_n:=\coprod\limits_{K_n}X
  \Bigr)
\end{align*}
the coproduct in $\M$, indexed over the set $K_n$, of copies of $X_n$ (resp. $X$). Let $z\geq 0$ and define the \emph{evaluation} functor $\functor{\Ev_z}{\sM}{\M}$ objectwise by $\Ev_z(X):=X_z$. 
\end{defn}

If $\M$ is any of the model categories defined above---or more generally, if $\M$ is a cofibrantly generated model category---then it is easy to check that the diagram category $\sM$ inherits a corresponding projective model category structure, where the weak equivalences (resp. fibrations) are the objectwise weak equivalences (resp. objectwise fibrations). In each case, the model structure on the diagram category $\sM$ is cofibrantly generated, and is created by the set of adjunctions
\begin{align*}
\xymatrix{
  \M\ar@<0.5ex>[r]^-{-\cdot\Delta[z]} & 
  \sM\ar@<0.5ex>[l]^-{\Ev_z}
},\quad z\geq 0,
\end{align*}
with left adjoints on top. Since the right adjoints $\Ev_z$ commute with filtered colimits, the smallness conditions needed for the (possibly transfinite) small object arguments are satisfied. We refer to these model structures by the names above (e.g., the \emph{positive flat stable} model structure on $\sAlgO$).

\begin{defn}
\label{defn:hocolim}
Let $\M$ be a cofibrantly generated model category. The \emph{homotopy colimit} functor $\hocolim\limits_{\Delta^\op}$ is the total left derived functor 
\begin{align*}
\xymatrix{
  \sM\ar[d]\ar[r]^{\colim\limits_{\Delta^\op}} & \M\ar[rr] && \Ho(\M)\\
  \Ho(\sM)\ar[rrr]^{\hocolim\limits_{\Delta^\op}}_{\text{total left derived functor}} &&& \Ho(\M)
}
\end{align*}
of the colimit functor $\sM\rarrow\M$. We will sometimes use the notation
$
  \hocolim\limits^\M_{\Delta^\op}:=\hocolim\limits_{\Delta^\op}
$
to emphasize the category $\M$.
\end{defn}

\begin{rem}
It is easy to check that the right adjoint $\M\rarrow\sM$ of the colimit functor preserves fibrations and acyclic fibrations, hence by \cite[9.7]{Dwyer_Spalinski} the homotopy colimit functor is well-defined.
\end{rem}

\subsection{Homotopy colimits commute with left {Q}uillen functors}

The purpose of this section is to prove Proposition \ref{prop:hocolim_commutes_with_left_quillen_M}, which verifies that homotopy colimits commute with left Quillen functors.

\begin{prop}
\label{prop:cofibrant_diagrams_are_objectwise_cofibrant}
Let $M$ be a cofibrantly generated model category. If $Z\in\sM$ is a cofibrant diagram, then $Z$ is objectwise cofibrant.
\end{prop}

\begin{proof}
Let $X\rarrow Y$ be a generating cofibration in $\M$, $z\geq 0$, and consider the pushout diagram
\begin{align}
\label{eq:simplicial_glueing_M}
\xymatrix{
  X\cdot\Delta[z]\ar[d]^{(*)}\ar[r] & Z_0\ar[d]^{(**)}\\
  Y\cdot\Delta[z]\ar[r] & Z_1
}
\end{align}
in $\sM$. Assume $Z_0$ is objectwise cofibrant; let's verify $Z_1$ is objectwise cofibrant. Since $(*)$ is objectwise a cofibration---the coproduct of a set of cofibrations in $\M$ is a cofibration---we know $(**)$ is objectwise a cofibration, and hence $Z_1$ is objectwise cofibrant. Consider a sequence
\begin{align}
\label{eq:transfinite_composition_M}
\xymatrix{
  Z_0\ar[r] & Z_1\ar[r] & Z_2\ar[r] & \cdots
}
\end{align}
of pushouts of maps as in \eqref{eq:simplicial_glueing_M}. Assume $Z_0$ is objectwise cofibrant; we want to show that $Z_\infty:=\colim_k Z_k$ is objectwise cofibrant. Since each map in \eqref{eq:transfinite_composition_M} is objectwise a cofibration, we know the induced map $Z_0\rarrow Z_\infty$ is objectwise a cofibration, and hence $Z_\infty$ is objectwise cofibrant. Noting that every cofibration $\emptyset\rarrow Z$ in $\sM$ is a retract of a (possibly transfinite) composition of pushouts of maps as in \eqref{eq:simplicial_glueing_M}, starting with $Z_0=\emptyset$, finishes the proof.
\end{proof}

\begin{prop}
\label{prop:hocolim_commutes_with_left_quillen_M}
Let $M$ and $M'$ be cofibrantly generated model categories. Consider any Quillen adjunction 
\begin{align*}
\xymatrix{
  \M\ar@<0.5ex>[r]^-{F} & 
  \M'\ar@<0.5ex>[l]^-{G}
}
\end{align*}
with left adjoint on top. If $X\in\sM$, then there is a zig-zag of weak equivalences
\begin{align*}
  \LL F\bigl(\hocolim\limits_{\Delta^\op}X\bigr)\wequiv
  \hocolim\limits_{\Delta^\op}\LL F(X)
\end{align*}
natural in $X$. Here, $\LL F$ is the total left derived functor of $F$.
\end{prop}

\begin{proof}
Consider any $X\in\sM$. The map $\emptyset\rarrow X$ factors functorially $\emptyset\rarrow X^c\rarrow X$ in $\sM$ as a cofibration followed by an acyclic fibration. This gives natural zig-zags of weak equivalences
\begin{align*}
  \LL F\bigl(\hocolim\limits_{\Delta^\op}X\bigr)&\wequiv
  \LL F\bigl(\hocolim\limits_{\Delta^\op}X^c\bigr)\wequiv
  \LL F\bigl(\colim\limits_{\Delta^\op}X^c\bigr)\wequiv
  F\bigl(\colim\limits_{\Delta^\op}X^c\bigr)\\
  &\Iso
  \colim\limits_{\Delta^\op}F(X^c)\wequiv
  \hocolim\limits_{\Delta^\op}F(X^c)\wequiv
  \hocolim\limits_{\Delta^\op}\LL F(X^c)\\
  &\wequiv
  \hocolim\limits_{\Delta^\op}\LL F(X)
\end{align*}
which follow immediately from Proposition \ref{prop:cofibrant_diagrams_are_objectwise_cofibrant} and the following observation: the right adjoint of the functor
\begin{align}
\label{eq:induced_left_quillen}
  \function{F}{\sM}{\sM'},&\quad\quad 
  Y\longmapsto F(Y)\quad\quad \bigl(n\longmapsto F(Y_n)\bigr),
\end{align}
preserves fibrations and acyclic fibrations, hence by \cite[9.7]{Dwyer_Spalinski} the functor \eqref{eq:induced_left_quillen} preserves cofibrant diagrams.
\end{proof}

\section{Homotopy colimits in the underlying categories}
\label{sec:hocolim_underlying}

The purpose of this section is to prove Theorems \ref{thm:hocolim_realzn} and \ref{thm:calculate_hocolim_with_realization_symseq_nice}, which calculate certain homotopy colimits in the underlying categories. The arguments provide a useful warm-up for proving Theorem \ref{MainTheorem3}, which calculates certain homotopy colimits in algebras and modules over operads (Section \ref{sec:hocolim_calculations_algebraic}).

\begin{assumption}
From now on in this section, we assume that $\unit$ is any commutative ring.
\end{assumption}

Denote by $\Mod_\unit$ the category of $\unit$-modules and by $\Chaincx_\unit^+$ the category of non-negative chain complexes over $\unit$. There are adjunctions
\begin{align*}
\xymatrix{
  \sSet\ar@<0.5ex>[r]^-{(-)_+} & 
  \sSet_*\ar@<0.5ex>[l]^-{U}\ar@<0.5ex>[r]^-{S\tensor G_0} &
  \Spectra,\ar@<0.5ex>[l]
}\quad\quad
\xymatrix{
  \sSet\ar@<0.5ex>[r]^-{\unit} & 
  \sMod_\unit\ar@<0.5ex>[r]^-{\NN}\ar@<0.5ex>[l]^-{U} & 
  \Chaincx_\unit^+\ar@<0.5ex>[r]\ar@<0.5ex>[l] & 
  \Chaincx_\unit,\ar@<0.5ex>[l]
}
\end{align*}
with left adjoints on top, $U$ the forgetful functor, $\NN$ the normalization functor (Definition \ref{defn:normalization_functor}) appearing in the Dold-Kan correspondence \cite[III.2]{Goerss_Jardine}, \cite[8.4]{Weibel}, and the right-hand functor on top the natural inclusion of categories. We will denote by $\functor{\NN\unit}{\sSet}{\Chaincx_\unit}$ the composition of the left adjoints on the right-hand side.

\begin{rem}
The functor $S\tensor G_0$ is left adjoint to ``evaluation at $0$''; the notation agrees with \cite{Harper_Spectra} and \cite[after 2.2.5]{Hovey_Shipley_Smith}. Let $X\in\Spectra$ and $K\in\sSet_*$. There are natural isomorphisms $X\Smash K\Iso X\Smash (S\tensor G_0 K)$ in $\Spectra$.
\end{rem}

\begin{rem}
\label{rem:pointed_and_unpointed_realization}
If $X\in\sSet_*$, there are natural isomorphisms $X\times_\Delta\Delta[-]\Iso X\Smash_\Delta\Delta[-]_+$.
\end{rem}

\begin{defn} 
\label{defn:realization}
The \emph{realization} functors $|-|$ for simplicial symmetric spectra and simplicial unbounded chain complexes over $\unit$ are defined objectwise by the coends
\begin{align*}
  \functor{|-|}{\sSpectra}{\Spectra},
  &\quad\quad
  X\longmapsto |X|:=X\Smash_{\Delta}\Delta[-]_+\ ,\\
  \functor{|-|}{\sChaincx_\unit}{\Chaincx_\unit},
  &\quad\quad
  X\longmapsto |X|:=X\tensor_{\Delta}\NN\unit\Delta[-].
\end{align*}
\end{defn}

\begin{prop}
\label{prop:realzns_fit_into_adjunctions}
The realization functors fit into adjunctions
\begin{align*}
\xymatrix{
  \sSpectra
  \ar@<0.5ex>[r]^-{|-|} & \Spectra,\ar@<0.5ex>[l]
}\quad\quad
\xymatrix{
  \sChaincx_\unit
  \ar@<0.5ex>[r]^-{|-|} & \Chaincx_\unit,\ar@<0.5ex>[l]
}
\end{align*}
with left adjoints on top. Each adjunction is a Quillen pair.
\end{prop}

\begin{proof}
Consider the case of $\sSpectra$ (resp. $\sChaincx_\unit$). Using the universal property of coends, it is easy to verify that the functor given objectwise by
\begin{align*}
  \Map(S\tensor G_0\Delta[-]_+,Y)\quad\quad
  \Bigl(
  \text{resp.}\quad
  \Map(\NN\unit\Delta[-],Y)
  \Bigr)
\end{align*}
is a right adjoint of $|-|$. To check the adjunctions form Quillen pairs, it is enough to verify the right adjoints preserve fibrations and acyclic fibrations; since the model structures on $\Spectra$ and $\Chaincx_\unit$ are monoidal model categories, this follows by noting that $S\tensor G_0\Delta[m]_+$ and $\NN\unit\Delta[m]$ are cofibrant for each $m\geq 0$.
\end{proof}

\begin{prop}
Let $X$ be a symmetric spectrum (resp. unbounded chain complex over $\unit$).  There are isomorphisms $|X\cdot\Delta[0]|\Iso X$, natural in $X$.
\end{prop}

\begin{proof}
This follows from uniqueness of left adjoints (up to isomorphism).
\end{proof}

\subsection{Homotopy colimits of simplicial objects in $\Spectra$ and $\Chaincx_\unit$}
\label{sec:sym_spectra_and_unbounded_chain_cxs}

The purpose of this section is to prove Theorem \ref{thm:hocolim_realzn}. A first step is to establish some of the good properties of realization and to recall the notion of simplicially homotopic maps. 

The following proposition is motivated by a similar argument given in \cite[IV.1.7]{Goerss_Jardine} and \cite[X.2.4]{EKMM} in the contexts of bisimplicial sets and proper simplicial spectra, respectively; see also \cite[A]{Dugger_Isaksen} and \cite[Chapter 18]{Hirschhorn} for related arguments. We defer the proof to Section \ref{sec:proofs}.

\begin{prop}\label{prop:realzn_homotopy_meaningful}
Let $\function{f}{X}{Y}$ be a morphism of simplicial symmetric spectra (resp. simplicial unbounded chain complexes over $\unit$). If $f$ is an objectwise weak equivalence, then $\function{|f|}{|X|}{|Y|}$ is a weak equivalence.
\end{prop}

We prove the following two propositions in Section \ref{sec:proofs}.

\begin{prop}\label{prop:realzn_preserves_monomorphisms}
Let $\function{f}{X}{Y}$ be a morphism of simplicial symmetric spectra (resp. simplicial unbounded chain complexes over $\unit$). If $f$ is a monomorphism,  then $\function{|f|}{|X|}{|Y|}$ is a monomorphism.
\end{prop}

\begin{prop}\label{prop:tot_of_normalization}
If $X$ is a simplicial unbounded chain complex over $\unit$, then there are isomorphisms
\begin{align*}
  |X|=X\tensor_\Delta\NN\unit\Delta[-]\Iso\Tot^\oplus\NN(X)
\end{align*}
natural in $X$.
\end{prop}

\begin{defn}\label{defn:simplicial_homotopy}
Let $\M$ be a category with all small colimits. Let $\function{f,g}{X}{Y}$ be maps in $\sM$ and consider the left-hand diagram 
\begin{align}
\label{eq:simplicial_homotopy}
\xymatrix{
  X\ar@<0.5ex>[r]^-{f}\ar@<-0.5ex>[r]_-{g} & Y,
}\quad\quad
\xymatrix{
  X\Iso X\cdot\Delta[0]\ar@<0.5ex>[r]^-{\id\cdot d^1}
  \ar@<-0.5ex>[r]_-{\id\cdot d^0} &
  X\cdot\Delta[1]\ar[r]^-{H} & Y,
}
\end{align}
in $\sM$. A \emph{simplicial homotopy} from $f$ to $g$ is a map $\function{H}{X\cdot\Delta[1]}{Y}$ in $\sM$ such that the two diagrams in \eqref{eq:simplicial_homotopy} are identical. The map $f$ is \emph{simplicially homotopic} to $g$ if there exists a simplicial homotopy from $f$ to $g$.
\end{defn}

\begin{rem}
This definition of simplicial homotopy agrees with \cite[I.6]{Goerss_Jardine} and \cite[between 6.2 and 6.3]{May_simplicial}. Consider the maps $\function{d^0,d^1}{\Delta[0]}{\Delta[1]}$ appearing in \eqref{eq:simplicial_homotopy}; it is important to note that the map $d^1$ represents the vertex 0 and the map $d^0$ represents the vertex 1. Hence this definition of simplicial homotopy is the intuitively correct one \cite[I.6]{Goerss_Jardine}, and is simply the reverse of the definition that is sometimes written down in terms of relations involving face and degeneracy maps: giving a simplicial homotopy $H$ from $f$ to $g$ as in Definition \ref{defn:simplicial_homotopy} is the same as giving a simplicial homotopy from $g$ to $f$ as defined in \cite[9.1]{May}, \cite[5.1]{May_simplicial}, \cite[8.3.11]{Weibel}.
\end{rem}

An easy proof of the following is given in \cite[Proof of I.7.10]{Goerss_Jardine}.

\begin{prop}
\label{prop:simplicial_contraction_simplices}
Let $z\geq 0$ and consider the maps
\begin{align*}
\xymatrix{
  \Delta[z]\ar[r]^{r} &
  \Delta[0]\ar[r]^{s} &
  \Delta[z]
}
\end{align*}
in $\sSet$ such that the map $s$ represents the vertex 0. Then the map $sr$ is simplicially homotopic to the identity map.
\end{prop}

\begin{prop}
\label{prop:simplicial_contraction_nice}
Let $W$ be a symmetric spectrum (resp. unbounded chain complex over $\unit$) and $z\geq 0$. Consider the maps
\begin{align*}
\xymatrix{
  W\cdot\Delta[z]\ar[r]^{\id\cdot r} &
  W\cdot\Delta[0]\ar[r]^{\id\cdot s} &
  W\cdot\Delta[z]
}
\end{align*}
in $\sSpectra$ (resp. $\sChaincx_\unit$) induced by the maps
$
\xymatrix@1{
  \Delta[z]\ar[r]^{r} & 
  \Delta[0]\ar[r]^{s} & 
  \Delta[z]
}
$
in simplicial sets, such that the map $s$ represents the vertex 0. Then the map
\begin{align*}
\xymatrix{
  |W\cdot\Delta[z]|\ar[r]^-{|\id\cdot r|} &
  |W\cdot\Delta[0]|\Iso W
}
\end{align*}
in $\Spectra$ (resp. in $\Chaincx_\unit$) is a weak equivalence. 
\end{prop}

\begin{proof}
We know that $rs=\id$ and $sr$ is simplicially homotopic to the identity map (Proposition \ref{prop:simplicial_contraction_simplices}). Hence $(\id\cdot r)(\id\cdot s)=\id$ and $(\id\cdot s)(\id\cdot r)$ is simplicially homotopic to the identity map. In the case of symmetric spectra, since every level equivalence is a weak equivalence, it follows that $|\id\cdot r|$ is a weak equivalence. In the case of chain complexes, since $\Tot^\oplus\NN$ takes simplicially homotopic maps to chain homotopic maps, it follows from Proposition \ref{prop:tot_of_normalization} that $|\id\cdot r|$ is a weak equivalence.
\end{proof}

\begin{prop}\label{prop:comparing_hocolim_and_realzn}
If $Z$ is a cofibrant simplicial symmetric spectrum (resp. cofibrant simplicial unbounded chain complex over $\unit$), then the natural map
\begin{align*}
\xymatrix{
  |Z|\ar[r] & 
  |(\pi_0Z)\cdot\Delta[0]|\Iso \pi_0Z
}
\end{align*}
is a weak equivalence. 
\end{prop}

\begin{proof}
Let $X\rarrow Y$ be a generating cofibration in $\Spectra$ (resp. $\Chaincx_\unit$) and $z\geq 0$. Consider the pushout diagram
\begin{align}
\label{eq:simplicial_glueing_chain_complex}
\xymatrix{
  X\cdot\Delta[z]\ar[d]\ar[r] & Z_0\ar[d]\\
  Y\cdot\Delta[z]\ar[r] & Z_1
}
\end{align}
in $\sSpectra$ (resp. $\sChaincx_\unit$) and the natural maps
\begin{align}
\label{eq:colim_map_first}
\xymatrix{
  |Z_0|\ar[r] & |(\pi_0Z_0)\cdot\Delta[0]|,
}\\
\label{eq:colim_map_second}
\xymatrix{
  |Z_1|\ar[r] & |(\pi_0Z_1)\cdot\Delta[0]|.
}
\end{align}
Assume \eqref{eq:colim_map_first} is a weak equivalence; let's verify \eqref{eq:colim_map_second} is a weak equivalence. Consider the commutative diagram
\begin{align*}
\xymatrix{
  |Z_0|\ar[d]\ar[r] & 
  |Z_1|\ar[d]\ar[r] & 
  |(Y/X)\cdot\Delta[z]|\ar[d] \\
  |(\pi_0Z_0)\cdot\Delta[0]|\ar[r] & 
  |(\pi_0Z_1)\cdot\Delta[0]|\ar[r] & 
  |(Y/X)\cdot\Delta[0]|
}
\end{align*}
The left-hand horizontal maps are monomorphisms, the left-hand vertical map is a weak equivalence by assumption, and the right-hand vertical map is a weak equivalence by Proposition \ref{prop:simplicial_contraction_nice}, hence the middle vertical map is a weak equivalence. Consider a sequence
\begin{align*}
\xymatrix{
  Z_0\ar[r] & Z_1\ar[r] & Z_2\ar[r] & \cdots
}
\end{align*}
of pushouts of maps as in \eqref{eq:simplicial_glueing_chain_complex}. Assume $Z_0$ makes \eqref{eq:colim_map_first} a weak equivalence; we want to show that for $Z_\infty:=\colim_k Z_k$ the natural map
\begin{align}
\label{eq:colim_map_infinity}
\xymatrix{
  |Z_\infty|\ar[r] & 
  |(\pi_0Z_\infty)\cdot\Delta[0]|
}
\end{align}
is a weak equivalence. Consider the commutative diagram
\begin{align*}
\xymatrix{
  |Z_0|\ar[d]\ar[r] & 
  |Z_1|\ar[d]\ar[r] & 
  |Z_2|\ar[d]\ar[r] & \cdots \\
  |(\pi_0Z_0)\cdot\Delta[0]|\ar[r] & 
  |(\pi_0Z_1)\cdot\Delta[0]|\ar[r] & 
  |(\pi_0Z_2)\cdot\Delta[0]|\ar[r] &
  \cdots
}
\end{align*}
We know that the horizontal maps are monomorphisms and the vertical maps are weak equivalences, hence the induced map \eqref{eq:colim_map_infinity} is a weak equivalence. Noting that every cofibration $*\rarrow Z$ in $\sSpectra$ (resp. $\sChaincx_\unit$) is a retract of a (possibly transfinite) composition of pushouts of maps as in \eqref{eq:simplicial_glueing_chain_complex}, starting with $Z_0=*$, finishes the proof.
\end{proof}

\begin{proof}[Proof of Theorem \ref{thm:hocolim_realzn}]
Consider any map $X\rarrow Y$ in $\sSpectra$ (resp. $\sChaincx_\unit$). Use functorial factorization to obtain a commutative diagram
\begin{align*}
\xymatrix{
  {*}\ar[d]\ar[r] & X^c\ar[d]\ar[r] & X\ar[d]\\
  {*}\ar[r] & Y^c\ar[r] & Y
}
\end{align*}
in $\sSpectra$ (resp. $\sChaincx_\unit$) such that each row is a cofibration followed by an acyclic fibration. Hence we get a corresponding commutative diagram
\begin{align*}
\xymatrix{
  \hocolim\limits_{\Delta^\op} X\ar[d] & 
  \hocolim\limits_{\Delta^\op} X^c\ar[d]\ar[r]\ar[l] &
  \colim\limits_{\Delta^\op} X^c\ar[d] &
  |X^c|\ar[d]\ar[l]_-{(*)}\ar[r]^-{(**)}&
  |X|\ar[d]\\
  \hocolim\limits_{\Delta^\op} Y & 
  \hocolim\limits_{\Delta^\op} Y^c\ar[l]\ar[r]&
  \colim\limits_{\Delta^\op} Y^c &
  |Y^c|\ar[l]_-{(*)}\ar[r]^-{(**)}&
  |Y|
}
\end{align*}
such that the rows are maps of weak equivalences; the maps $(*)$ and $(**)$ are weak equivalences by Propositions \ref{prop:comparing_hocolim_and_realzn} and \ref{prop:realzn_homotopy_meaningful}, respectively.
\end{proof}

\subsection{Homotopy colimits of simplicial objects in $\SymSeq$}

The purpose of this section is to observe that several properties of realization proved in Section \ref{sec:sym_spectra_and_unbounded_chain_cxs} have corresponding versions for simplicial symmetric sequences.

\begin{defn}\label{defn:realzn_SymSeq}
Consider symmetric sequences in $\Spectra$ or $\Chaincx_\unit$. The \emph{realization} functor $|-|$ for simplicial symmetric sequences is defined objectwise by
\begin{align*}
  \functor{|-|}{\sSymSeq}{\SymSeq},
  \quad\quad
  X\longmapsto |X|
  \quad\quad
  \Bigl(
  \mathbf{t}\longmapsto |X[\mathbf{t}]|
  \Bigr).
\end{align*}
\end{defn}

The following is a version of Proposition \ref{prop:realzns_fit_into_adjunctions} for simplicial symmetric sequences, and is proved by exactly the same argument.

\begin{prop}
Consider symmetric sequences in $\Spectra$ or $\Chaincx_\unit$. The realization functor fits into an adjunction
\begin{align*}
\xymatrix{
  \sSymSeq
  \ar@<0.5ex>[r]^-{|-|} & \SymSeq\ar@<0.5ex>[l]
}
\end{align*}
with left adjoint on top. This adjunction is a Quillen pair.
\end{prop}

\begin{prop}\label{prop:realzn_homotopy_meaningful_SymSeq}
Consider symmetric sequences in $\Spectra$ or $\Chaincx_\unit$.
\begin{itemize}
\item[(a)] If $\function{f}{X}{Y}$ in $\sSymSeq$ is a weak equivalence, then $\function{|f|}{|X|}{|Y|}$ is a weak equivalence.
\item[(b)] If $\function{f}{X}{Y}$ in $\sSymSeq$ is a monomorphism, then $\function{|f|}{|X|}{|Y|}$ is a monomorphism.
\item[(c)] If $Z\in\sSymSeq$ is cofibrant, then the natural map
\begin{align*}
\xymatrix{
  |Z|\ar[r] & 
  |(\pi_0Z)\cdot\Delta[0]|\Iso \pi_0Z
}
\end{align*}
is a weak equivalence.
\end{itemize}
\end{prop}

\begin{proof}
Parts (a) and (b) follow immediately from Propositions \ref{prop:realzn_homotopy_meaningful} and \ref{prop:realzn_preserves_monomorphisms}, respectively. Part (c) is a version of Proposition \ref{prop:comparing_hocolim_and_realzn} for simplicial symmetric sequences, and is proved by exactly the same argument.
\end{proof}

\begin{thm}
\label{thm:calculate_hocolim_with_realization_symseq_nice}
Consider symmetric sequences in $\Spectra$ or $\Chaincx_\unit$. If $X$ is a simplicial symmetric sequence, then there is a zig-zag of weak equivalences 
\begin{align*}
  \hocolim\limits_{\Delta^\op} X\wequiv |X|
\end{align*}
natural in $X$. Here, $\unit$ is any commutative ring.
\end{thm}

\begin{proof}
This is proved exactly as Theorem \ref{thm:hocolim_realzn}, except using Proposition \ref{prop:realzn_homotopy_meaningful_SymSeq} instead of Propositions \ref{prop:comparing_hocolim_and_realzn} and \ref{prop:realzn_homotopy_meaningful}.
\end{proof}

\section{Homotopy colimits in algebras and modules over operads}
\label{sec:hocolim_calculations_algebraic}

The purpose of this section is to prove Theorems \ref{MainTheorem3}, \ref{thm:fattened_replacement}, and \ref{thm:bar_calculates_derived_circle}. Since certain properties of reflexive coequalizers will be important, we first recall these in Section \ref{sec:reflexive_coequalizers}. In Section \ref{sec:filtrations_of_certain_pushouts} we introduce filtrations of certain pushouts in (simplicial) algebras and modules over operads, which are a key ingredient in the proofs.

\begin{assumption}
From now on in this section, we assume that $\unit$ is any commutative ring, unless stated otherwise.
\end{assumption}

Later in this section we will need the stronger assumption that $\unit$ is a field of characteristic zero when we begin using model structures on algebras and modules over operads in unbounded chain complexes (Definition \ref{defn:model_structures_chain_complexes}).

\subsection{Reflexive coequalizers and colimits of simplicial objects}
\label{sec:reflexive_coequalizers}
A first step is to recall the good behavior of reflexive coequalizers with respect to tensor products and circle products. 

\begin{defn}
Let $\C$ be a category. A pair of maps of the form $\xymatrix@1{X_0 & X_1\ar@<-0.5ex>[l]_-{d_0}\ar@<+0.5ex>[l]^-{d_1}}$ in $\C$ is called a \emph{reflexive pair} if there exists $\function{s_0}{X_0}{X_1}$ in $\C$ such that $d_0s_0=\id$ and $d_1s_0=\id$. A \emph{reflexive coequalizer} is the coequalizer of a reflexive pair.
\end{defn}

The following proposition is proved in \cite[2.3.2 and 2.3.4]{Rezk}; it is also proved in \cite{Harper_Modules} and follows from the proof of \cite[II.7.2]{EKMM} or the arguments in \cite[Section 1]{Goerss_Hopkins}.

\begin{prop}
\label{prop:well_behaved_wrt_tensor_and_circle}
Let $(\C,\tensor)$ be a closed symmetric monoidal category with all small colimits. Consider symmetric sequences in $\Spectra$ or $\Chaincx_\unit$. 
\begin{itemize}
\item[(a)] If
$\xymatrix@1{X_{-1} & X_0\ar[l] & X_1\ar@<-0.5ex>[l]\ar@<+0.5ex>[l]}$ and 
$\xymatrix@1{Y_{-1} & Y_0\ar[l] & Y_1\ar@<-0.5ex>[l]\ar@<+0.5ex>[l]}$
are reflexive coequalizer diagrams in $\C$, then their objectwise tensor product
\begin{align*}
\xymatrix{
  X_{-1}\tensor Y_{-1} & X_0\tensor Y_0\ar[l]& 
  X_1\tensor Y_1\ar@<-0.5ex>[l]\ar@<+0.5ex>[l]
}
\end{align*}
is a reflexive coequalizer diagram in $\C$.
\item[(b)]  If
$\xymatrix@1{A_{-1} & A_0\ar[l] & A_1\ar@<-0.5ex>[l]\ar@<+0.5ex>[l]}$ and 
$\xymatrix@1{B_{-1} & B_0\ar[l] & B_1\ar@<-0.5ex>[l]\ar@<+0.5ex>[l]}$
are reflexive coequalizer diagrams in $\SymSeq$, then their objectwise circle product 
\begin{align*}
\xymatrix@1{
  A_{-1}\circ B_{-1} & A_0\circ B_0\ar[l]& 
  A_1\circ B_1\ar@<-0.5ex>[l]\ar@<+0.5ex>[l]
}
\end{align*}
is a reflexive coequalizer diagram in $\SymSeq$. 
\end{itemize}
\end{prop}

The following relationship between reflexive coequalizers and simplicial objects will be useful.

\begin{prop}
\label{prop:colim_of_simplicial_object}
Let $\M$ be a category with all small colimits. If $X\in\sM$, then its colimit is naturally isomorphic to a reflexive coequalizer of the form
\begin{align*}
  \colim_{\Delta^\op}X\Iso
  \colim\Bigl(
  \xymatrix{
    X_0 & X_1
    \ar@<-0.5ex>[l]_-{d_0}\ar@<0.5ex>[l]^-{d_1}
  }
  \Bigr)
\end{align*}
in $\M$, with $d_0$ and $d_1$ the indicated face maps of $X$.
\end{prop}

\begin{proof}
This follows easily by using the simplicial identities \cite[I.1]{Goerss_Jardine} to verify the universal property of colimits.
\end{proof}

\begin{prop}
\label{prop:colimits_of_simplicial_objects_and_forgetful}
Let $\capO$ be an operad in symmetric spectra or unbounded chain complexes over $\unit$. If $X$ is a simplicial $\capO$-algebra (resp. simplicial left $\capO$-module), then there are isomorphisms
\begin{align*}
  U\colim\limits^{\AlgO}_{\Delta^\op}X  \Iso 
  \colim\limits_{\Delta^\op}U X \quad\quad
  \Bigl(
  \text{resp.}\quad
  U\colim\limits^{\LtO}_{\Delta^\op}X  \Iso 
  \colim\limits_{\Delta^\op}U X
  \Bigr)
\end{align*}
natural in $X$. Here, $U$ is the forgetful functor.
\end{prop}

\begin{proof}
This follows immediately from Propositions \ref{prop:colim_of_simplicial_object} and \ref{prop:basic_properties_LtO}.
\end{proof}

\subsection{Filtrations of certain pushouts}
\label{sec:filtrations_of_certain_pushouts}
The purpose of this section is to observe that several constructions and propositions proved in \cite{Harper_Spectra} have corresponding objectwise versions for $\Delta^\op$-shaped diagrams in algebras and modules over operads; the resulting filtrations will be important to several results in this paper.

\begin{defn}
\label{defn:objectwise_tensor_product_for_diagrams}
Let $(\C,\tensor)$ be a monoidal category. If $X,Y\in\sC$ then $X\tensor Y\in\sC$ is defined objectwise by
\begin{align*}
  (X\tensor Y)_n:=X_n\tensor Y_n.
\end{align*} 
\end{defn}

\begin{defn}\label{def:symmetric_array}
Let $\capO$ be an operad in symmetric spectra (resp. unbounded chain complexes over $\unit$) and consider symmetric sequences in $\Spectra$ (resp. $\Chaincx_\unit$). 
\begin{itemize}
\item If $Y\in\sSymSeq$, then $\capO\circ Y\in\sLtO$ denotes the composition of functors $\Delta^\op\xrightarrow{Y}\SymSeq\xrightarrow{\capO\circ-}\LtO$.
\item A \emph{symmetric array} in $\Spectra$ (resp. $\Chaincx_\unit$) is a symmetric sequence in $\SymSeq$; i.e. a functor $\functor{A}{\Sigma^\op}{\SymSeq}$.
\item $\SymArray:=\SymSeq^{\Sigma^\op}$ is the category of symmetric arrays in $\Spectra$ (resp. $\Chaincx_\unit$) and their natural transformations.
\end{itemize}
\end{defn}

The following proposition follows from \cite{Harper_Spectra} as indicated below.

\begin{prop}\label{prop:coproduct_modules}
Let $\capO$ be an operad in symmetric spectra or unbounded chain complexes over $\unit$, $A\in\sLtO$ (resp. $A\in\LtO$), and $Y\in\sSymSeq$ (resp. $Y\in\SymSeq$). Consider any coproduct in $\sLtO$ (resp. $\LtO$) of the form
\begin{align}\label{eq:coproduct_diagram_modules}
\xymatrix{
  A\amalg(\capO\circ Y).
}
\end{align}  
There exists $\capO_A\in\sSymArray$ (resp. $\capO_A\in\SymArray$) and natural isomorphisms
\begin{align*}
  A\amalg(\capO\circ Y) \Iso 
  \coprod\limits_{q\geq 0}\capO_A[\mathbf{q}]
  \tensorcheck_{\Sigma_q}Y^{\tensorcheck q}
\end{align*}
in the underlying category $\sSymSeq$ (resp. $\SymSeq$).
\end{prop}

\begin{rem}
Other possible notations for $\capO_A$ include $\U_\capO(A)$ or $\U(A)$; these are closer to the notation used in \cite{Elmendorf_Mandell, Mandell} and are not to be confused with the forgetful functors.
\end{rem}

\begin{proof}[Proof of Proposition \ref{prop:coproduct_modules}]
Consider the case of symmetric spectra. The case of $A\in\LtO$ and $Y\in\SymSeq$ is proved in \cite{Harper_Spectra}, and the case of $A\in\sLtO$ and $Y\in\sSymSeq$ is proved in exactly the same way, except using the obvious objectwise construction of $\capO_A$. The case of unbounded chain complexes over $\unit$ is similar.
\end{proof}

\begin{prop}\label{prop:simplicial_colimit_commutes}
Let $\capO$ be an operad in symmetric spectra or unbounded chain complexes over $\unit$, $A\in\sLtO$, $Y\in\sSymSeq$, and $t\geq 0$. There are natural isomorphisms
\begin{align}
  \label{eq:simplicial_colimit_commutes}
  \colim\limits_{\Delta^\op}^{\LtO}
  \Bigl(
  A\amalg(\capO\circ Y)
  \Bigr)
  &\Iso 
  \coprod\limits_{q\geq 0}\capO_{\pi_0 A}[\mathbf{q}]
  \tensorcheck_{\Sigma_q}(\pi_0 Y)^{\tensorcheck q},\\
  \label{eq:simplicial_colimit_commutes_array}
  \colim\limits_{\Delta^\op}
  \Bigl( 
  \capO_A[\mathbf{t}] 
  \Bigr)
  &\Iso
  \capO_{\pi_0 A}[\mathbf{t}],
\end{align}
in the underlying category $\SymSeq$. 
\end{prop}

\begin{proof}
The isomorphism in \eqref{eq:simplicial_colimit_commutes} follows from the natural isomorphisms
\begin{align*}
  \colim\limits_{\Delta^\op}^{\LtO}\Bigl(A\amalg(\capO\circ Y)\Bigr)
  \Iso (\pi_0 A)\amalg\pi_0(\capO\circ Y)
  \Iso (\pi_0 A)\amalg\bigl(\capO\circ(\pi_0 Y)\bigr)
\end{align*}
in $\LtO$ together with Proposition \ref{prop:coproduct_modules}, and the isomorphism in \eqref{eq:simplicial_colimit_commutes_array} follows easily from Propositions \ref{prop:colim_of_simplicial_object} and \ref{prop:well_behaved_wrt_tensor_and_circle}.
\end{proof}

\begin{defn}\label{def:filtration_setup_modules}
Let $\function{i}{X}{Y}$ be a morphism in $\sSymSeq$ (resp. $\SymSeq$) and $t\geq 1$. Define $Q_0^t:=X^{\tensorcheck t}$ and $Q_t^t:=Y^{\tensorcheck t}$. For $0<q<t$ define $Q_q^t$ inductively by the pushout diagrams
\begin{align*}
\xymatrix{
  \Sigma_t\cdot_{\Sigma_{t-q}\times\Sigma_{q}}X^{\tensorcheck(t-q)}
  \tensorcheck Q_{q-1}^q\ar[d]^{i_*}\ar[r]^-{\pr_*} & Q_{q-1}^t\ar[d]\\
  \Sigma_t\cdot_{\Sigma_{t-q}\times\Sigma_{q}}X^{\tensorcheck(t-q)}
  \tensorcheck Y^{\tensorcheck q}\ar[r] & Q_q^t
}
\end{align*}
in $\sSymSeq^{\Sigma_t}$ (resp. $\SymSeq^{\Sigma_t}$). The maps $\pr_*$ and $i_*$ are the obvious maps induced by $i$ and the appropriate projection maps.
\end{defn}

The following proposition follows from \cite{Harper_Spectra} as indicated below.

\begin{prop}\label{prop:small_arg_pushout_modules}
Let $\capO$ be an operad in symmetric spectra or unbounded chain complexes over $\unit$, $A\in\sLtO$ (resp. $A\in\LtO$), and $\function{i}{X}{Y}$ in $\sSymSeq$ (resp. in $\SymSeq$). Consider any pushout diagram in $\sLtO$ (resp. $\LtO$) of the form
\begin{align}\label{eq:small_arg_pushout_modules}
\xymatrix{
  \capO\circ X\ar[r]^-{f}\ar[d]^{\id\circ i} & A\ar[d]^{j}\\
  \capO\circ Y\ar[r] & A\amalg_{(\capO\circ X)}(\capO\circ Y)
}
\end{align}
The pushout in \eqref{eq:small_arg_pushout_modules} is naturally isomorphic to a filtered colimit of the form
\begin{align}\label{eq:filtered_colimit_modules}
  A\amalg_{(\capO\circ X)}(\capO\circ Y)\Iso 
  \colim\Bigl(
  \xymatrix{
    A_0\ar[r]^{j_1} & A_1\ar[r]^{j_2} & A_2\ar[r]^{j_3} & \dotsb
  }
  \Bigr)
\end{align}
in the underlying category $\sSymSeq$ (resp. $\SymSeq$), with $A_0:=\capO_A[\mathbf{0}]\Iso A$ and $A_t$ defined inductively by pushout diagrams in $\sSymSeq$ (resp. $\SymSeq$) of the form
\begin{align}\label{eq:good_filtration_modules}
\xymatrix{
  \capO_A[\mathbf{t}]\tensorcheck_{\Sigma_t}Q_{t-1}^t\ar[d]^{\id\tensorcheck_{\Sigma_t}i_*}
  \ar[r]^-{f_*} & A_{t-1}\ar[d]^{j_t}\\
  \capO_A[\mathbf{t}]\tensorcheck_{\Sigma_t}Y^{\tensorcheck t}\ar[r]^-{\xi_t} & A_t
}
\end{align}
\end{prop}

\begin{proof}
Consider the case of symmetric spectra. The case of $A\in\LtO$ and $\function{i}{X}{Y}$ in $\SymSeq$ is proved in \cite{Harper_Spectra}, and the case of $A\in\sLtO$ and $\function{i}{X}{Y}$ in $\sSymSeq$ is proved in exactly the same way, except using Proposition \ref{prop:coproduct_modules} and the obvious objectwise construction of the pushout diagrams \eqref{eq:good_filtration_modules}. The case of unbounded chain complexes over $\unit$ is similar.
\end{proof}

\begin{prop}\label{prop:colim_commutes_with_filtration}
Let $\capO$ be an operad in symmetric spectra or unbounded chain complexes over $\unit$, $A\in\sLtO$, and $\function{i}{X}{Y}$ in $\sSymSeq$. Consider any pushout diagram in $\sLtO$ of the form \eqref{eq:small_arg_pushout_modules}. Then $\pi_0(-)$ commutes with the filtered diagrams in \eqref{eq:filtered_colimit_modules}; i.e., there are natural isomorphisms which make the diagram
\begin{align*}
\xymatrix{
  \pi_0(A_0)\ar[d]^{\iso}\ar[r]^-{\pi_0(j_1)} & 
  \pi_0(A_1)\ar[d]^{\iso}\ar[r]^-{\pi_0(j_2)} & 
  \pi_0(A_2)\ar[d]^{\iso}\ar[r]^-{\pi_0(j_3)} & \dotsb\\
  (\pi_0 A)_0\ar[r]^-{j_1} & 
  (\pi_0 A)_1\ar[r]^-{j_2} & 
  (\pi_0 A)_2\ar[r]^-{j_3} & \dotsb
}
\end{align*}
commute. 
\end{prop}

\begin{proof}
This follows easily from Propositions \ref{prop:simplicial_colimit_commutes}, \ref{prop:colim_of_simplicial_object}, and \ref{prop:well_behaved_wrt_tensor_and_circle}.
\end{proof}

\subsection{Homotopy colimits of simplicial objects in $\AlgO$ and $\LtO$}
\label{sec:hocolim_simplicial_objects_algebras}
The purpose of this section is to prove Theorem \ref{MainTheorem3}, which can be understood as a homotopical version of Proposition \ref{prop:colimits_of_simplicial_objects_and_forgetful}.

\begin{prop}\label{prop:simplicial contraction}
Let $\capO$ be an operad in symmetric spectra (resp. unbounded chain complexes over $\unit$). Let $\mathcal{A}$ be a set, $W_\alpha\in\SymSeq$, and $n_\alpha\geq 0$, for each $\alpha\in\mathcal{A}$. Consider the maps
\begin{align*}
\xymatrix{
  \coprod\limits_{\alpha\in\mathcal{A}}
  (\capO\circ W_\alpha\cdot\Delta[n_\alpha])\ar@<0.5ex>[r]^{r} &
  \coprod\limits_{\alpha\in\mathcal{A}}
  (\capO\circ W_\alpha\cdot\Delta[0])\ar@<0.5ex>[r]^{s} &
  \coprod\limits_{\alpha\in\mathcal{A}}
  (\capO\circ W_\alpha\cdot\Delta[n_\alpha])
}
\end{align*}
in $\sLtO$ induced by the maps
$
\xymatrix@1{
  \Delta[n_\alpha]\ar[r]^{r_\alpha} & 
  \Delta[0]\ar[r]^{s_\alpha} & 
  \Delta[n_\alpha]
}
$
in simplicial sets, such that each map $s_\alpha$ represents the vertex 0. Then the map
\begin{align*}
\xymatrix{
  |\coprod\limits_{\alpha\in\mathcal{A}}
  (\capO\circ W_\alpha\cdot\Delta[n_\alpha])|\ar@<0.5ex>[r]^{|r|} &
  |\coprod\limits_{\alpha\in\mathcal{A}}
  (\capO\circ W_\alpha\cdot\Delta[0])|
}
\end{align*}
in $\SymSeq$ is a weak equivalence. 
\end{prop}

\begin{proof}
For each $\alpha\in\mathcal{A}$, we know that $r_\alpha s_\alpha=\id$ and $s_\alpha r_\alpha$ is simplicially homotopic to the identity map (Proposition \ref{prop:simplicial_contraction_simplices}). Hence $rs=\id$ and $sr$ is simplicially homotopic to the identity map. In the case of symmetric spectra, since every level equivalence is a weak equivalence, it follows that $|r|$ is a weak equivalence. In the case of chain complexes, since $\Tot^\oplus\NN$ takes simplicially homotopic maps to chain homotopic maps, it follows from Proposition \ref{prop:tot_of_normalization} that $|r|$ is a weak equivalence.
\end{proof}

\begin{prop}\label{prop:comparing_realzn_with_hocolim_LtO}
Let $\capO$ be an operad in symmetric spectra or unbounded chain complexes over $\unit$. If $Z$ is a cofibrant simplicial $\capO$-algebra (resp. cofibrant simplicial left $\capO$-module), then the natural map
\begin{align*}
\xymatrix{
  |U Z|\ar[r] & |(\pi_0 U Z)\cdot\Delta[0]|\Iso \pi_0 U Z
}
\end{align*}
is a weak equivalence. Here, $U$ is the forgetful functor and $\unit$ is any  field of characteristic zero.
\end{prop}

\begin{proof}
Let $X\rarrow Y$ be a generating cofibration in $\SymSeq$, $z\geq 0$, and consider the pushout diagram
\begin{align}
\label{eq:glueing_on_cells}
\xymatrix{
  \capO\circ X\cdot\Delta[z]\ar[r]\ar[d] & Z_0\ar[d]\\
  \capO\circ Y\cdot\Delta[z]\ar[r] & Z_1
}
\end{align}
in $\sLtO$. For each cofibrant $W_\alpha\in\SymSeq$, $n_\alpha\geq 0$, and set $\mathcal{A}$, consider the natural maps 
\begin{align}
  \label{eq:first_natural_map}
  &|Z_0\amalg\coprod\limits_{\alpha\in\mathcal{A}}
  (\capO\circ W_\alpha\cdot\Delta[n_\alpha])|\rarrow 
  |\bigl((\pi_0Z_0)\amalg\coprod\limits_{\alpha\in\mathcal{A}}
  \capO\circ W_\alpha\bigr)\cdot\Delta[0]|,\\
  \label{eq:second_natural_map}
  &|Z_1\amalg\coprod\limits_{\alpha\in\mathcal{A}}
  (\capO\circ W_\alpha\cdot\Delta[n_\alpha])|\rarrow 
  |\bigl((\pi_0Z_1)\amalg\coprod\limits_{\alpha\in\mathcal{A}}
  \capO\circ W_\alpha\bigr)\cdot\Delta[0]|,
\end{align}
and note that the diagram
\begin{align*}
\xymatrix{
  \capO\circ X\cdot\Delta[z]\ar[r]\ar[d] & 
  Z_0\amalg\coprod\limits_{\alpha\in\mathcal{A}}
  (\capO\circ W_\alpha\cdot\Delta[n_\alpha])=:A\ar[d]\\
  \capO\circ Y\cdot\Delta[z]\ar[r] & 
  Z_1\amalg\coprod\limits_{\alpha\in\mathcal{A}}
  (\capO\circ W_\alpha\cdot\Delta[n_\alpha])\Iso A_\infty
}
\end{align*}
is a pushout diagram in $\sLtO$. Assume \eqref{eq:first_natural_map} is a weak equivalence for each cofibrant $W_\alpha\in\SymSeq$, $n_\alpha\geq 0$, and set $\mathcal{A}$; let's verify \eqref{eq:second_natural_map} is a weak equivalence for each cofibrant $W_\alpha\in\SymSeq$, $n_\alpha\geq 0$, and set $\mathcal{A}$. Suppose $\mathcal{A}$ is a set, $W_\alpha\in\SymSeq$ is cofibrant, and $n_\alpha\geq 0$, for each $\alpha\in\mathcal{A}$. By Proposition \ref{prop:small_arg_pushout_modules} there are corresponding filtrations together with induced maps $\xi_t$ $(t\geq 1)$ which make the diagram
\begin{align*}
\xymatrix{
  |A_0|\ar[r]\ar[d]^{|\xi_0|} & 
  |A_1|\ar[r]\ar@{.>}[d]^{|\xi_1|} & 
  |A_2|\ar[r]\ar@{.>}[d]^{|\xi_2|} & \dotsb \\
  |(\pi_0A_0)\cdot\Delta[0]|\ar[r] & 
  |(\pi_0A_1)\cdot\Delta[0]|\ar[r] & 
  |(\pi_0A_2)\cdot\Delta[0]|\ar[r] & \dotsb
}
\end{align*}
in $\SymSeq$ commute. Since $|-|$ commutes with colimits we get
\begin{align*}
\xymatrix{
  \colim_t|A_t|\ar[r]^-{\Iso}\ar[d] & |A_\infty|\ar[d]\\
  \colim_t|(\pi_0A_t)\cdot\Delta[0]|\ar[r]^-{\Iso} &
  |(\pi_0A_\infty)\cdot\Delta[0]|
}
\end{align*}
By assumption we know that $|\xi_0|$ is a weak equivalence, and to verify \eqref{eq:second_natural_map} is a weak equivalence, it is enough to check that $|\xi_t|$ is a weak equivalence for each $t\geq 1$. Since the horizontal maps are monomorphisms and we know that there are natural isomorphisms
\begin{align*}
  |A_{t}|/|A_{t-1}|
  &\Iso|\capO_A[\mathbf{t}]  
  \tensorcheck_{\Sigma_t}(Y/X\cdot\Delta[z])^{\tensorcheck t}|,\\
  |(\pi_0A_t)\cdot\Delta[0]|/
  |(\pi_0A_{t-1})\cdot\Delta[0]|
  &\Iso|(\capO_{\pi_0A}[\mathbf{t}]  
  \tensorcheck_{\Sigma_t}(Y/X)^{\tensorcheck t})\cdot\Delta[0]|,
\end{align*} 
it is enough to verify that
\begin{align*}
\xymatrix{
  |A\amalg\capO\circ\bigl((Y/X)\cdot\Delta[z]\bigr)|\ar[r] & 
  |\bigl((\pi_0 A)\amalg \capO\circ(Y/X)\bigr)
  \cdot\Delta[0]|
}  
\end{align*}
is a weak equivalence. Noting that $Y/X$ is cofibrant finishes the argument that \eqref{eq:second_natural_map} is a weak equivalence. Consider a sequence
\begin{align*}
\xymatrix{
  Z_0\ar[r] & Z_1\ar[r] & Z_2\ar[r] & \dotsb
}
\end{align*}
of pushouts of maps as in \eqref{eq:glueing_on_cells}. Assume $Z_0$ makes \eqref{eq:first_natural_map} a weak equivalence for each cofibrant $W_\alpha\in\SymSeq$, $n_\alpha\geq 0$, and set $\mathcal{A}$;  we want to show that for $Z_\infty:=\colim_k Z_k$ the natural map
\begin{align}
  \label{eq:Z_infty_map}
  |Z_\infty\amalg\coprod\limits_{\alpha\in\mathcal{A}}
  (\capO\circ W_\alpha\cdot\Delta[n_\alpha])|\rarrow 
  |\bigl((\pi_0Z_\infty)\amalg\coprod\limits_{\alpha\in\mathcal{A}}
  \capO\circ W_\alpha\bigr)\cdot\Delta[0]|
\end{align}
is a weak equivalence for each cofibrant $W_\alpha\in\SymSeq$, $n_\alpha\geq 0$, and set $\mathcal{A}$. Consider the diagram
\begin{align*}
\xymatrix{
  |Z_0\amalg\coprod\limits_{\alpha\in\mathcal{A}}
  (\capO\circ W_\alpha\cdot\Delta[n_\alpha])|\ar[d]\ar[r] &
  |Z_1\amalg\coprod\limits_{\alpha\in\mathcal{A}}
  (\capO\circ W_\alpha\cdot\Delta[n_\alpha])|\ar[d]\ar[r] & \dotsb\\
  |\bigl((\pi_0Z_0)\amalg\coprod\limits_{\alpha\in\mathcal{A}}
  \capO\circ W_\alpha\bigr)\cdot\Delta[0]|\ar[r] &
  |\bigl((\pi_0Z_1)\amalg\coprod\limits_{\alpha\in\mathcal{A}}
  \capO\circ W_\alpha\bigr)\cdot\Delta[0]|\ar[r] & \dotsb
}
\end{align*}
in $\SymSeq$. The horizontal maps are monomorphisms and the vertical maps are weak equivalences, hence the induced map \eqref{eq:Z_infty_map} is a weak equivalence. Noting that every cofibration $\capO\circ*\cdot\Delta[0]\rarrow Z$ in $\sLtO$ is a retract of a (possibly transfinite) composition of pushouts of maps as in \eqref{eq:glueing_on_cells}, starting with $Z_0=\capO\circ*\cdot\Delta[0]$, together with Proposition~\ref{prop:simplicial contraction}, finishes the proof.
\end{proof}

\begin{proof}[Proof of Theorem \ref{MainTheorem3}]
Consider any $X\in\sAlgO$ (resp. $X\in\sLtO$). The map $\emptyset\rarrow X$ factors functorially $\emptyset\rarrow X^c\rarrow X$ in $\sAlgO$ (resp. $\sLtO$) as a cofibration followed by an acyclic fibration. This gives a diagram
\begin{align*}
&\xymatrix{
    \hocolim\limits^\AlgO_{\Delta^\op} X & 
    \hocolim\limits^\AlgO_{\Delta^\op} X^c\ar[r]\ar[l] &
    \colim\limits_{\Delta^\op} X^c &
    |X^c|\ar[l]_-{(*)}\ar[r]^-{(**)}&
    |X|
  }\\
\Bigl(
  \text{resp.}\quad
  &\xymatrix{
  \hocolim\limits^\LtO_{\Delta^\op} X & 
  \hocolim\limits^\LtO_{\Delta^\op} X^c\ar[r]\ar[l] &
  \colim\limits_{\Delta^\op} X^c &
  |X^c|\ar[l]_-{(*)}\ar[r]^-{(**)}&
  |X|
  }
\Bigr)
\end{align*}
of weak equivalences in the underlying category. The map $(*)$ is a weak equivalence by Proposition \ref{prop:comparing_realzn_with_hocolim_LtO} and the map $(**)$ is a weak equivalence by Proposition \ref{prop:realzn_homotopy_meaningful} (resp. Proposition \ref{prop:realzn_homotopy_meaningful_SymSeq}). Hence there is a zig-zag of weak equivalences
\begin{align*}
  U\hocolim\limits^{\AlgO}_{\Delta^\op}X &
  \wequiv |U X| \wequiv
  \hocolim\limits_{\Delta^\op}U X\\
  \Bigl(
  \text{resp.}\quad
  U\hocolim\limits^{\LtO}_{\Delta^\op}X &
  \wequiv |U X| \wequiv
  \hocolim\limits_{\Delta^\op}U X
\Bigr)
\end{align*}
natural in $X$, with $U$ the forgetful functor; the right-hand weak equivalence is Theorem \ref{thm:hocolim_realzn} (resp. Theorem \ref{thm:calculate_hocolim_with_realization_symseq_nice}).
\end{proof}

\subsection{Homotopy colimits and simplicial bar constructions}\label{sec:proof_of_main_theorem}

The purpose of this section is to prove Theorems \ref{thm:fattened_replacement} and \ref{thm:bar_calculates_derived_circle}. A first step is to recall that simplicial bar constructions arise whenever one has objects equipped with actions of a monoid object. A particular instance of this is the following.

\begin{defn}
\label{defn:simplicial_bar_constructions}
Let $\capO$ be an operad in symmetric spectra (resp. unbounded chain complexes over $\unit$), $X$ a right $\capO$-module, and $Y$ an $\capO$-algebra. The \emph{simplicial bar construction} (or two-sided simplicial bar construction) $\BAR(X,\capO,Y)$ in $\sSpectra$ (resp. $\sChaincx_\unit$) looks like (showing only the face maps)
\begin{align*}
\xymatrix{
  X\circ (Y) & 
  X\circ\capO\circ (Y)\ar@<-0.5ex>[l]_-{m\circ(\id)}
  \ar@<0.5ex>[l]^-{\id\circ (m)} & 
  X\circ\capO\circ\capO\circ (Y)\ar@<-1.0ex>[l]\ar[l]\ar@<1.0ex>[l] &
  \cdots\ar@<1.5ex>[l]\ar@<0.5ex>[l]\ar@<-0.5ex>[l]\ar@<-1.5ex>[l]
}
\end{align*}
and is defined objectwise by $\BAR(X,\capO,Y)_k:=X\circ\capO^{\circ k}\circ (Y)$ with the obvious face and degeneracy maps induced by the multiplication and unit maps \cite[A.1]{Gugenheim_May}, \cite[Section 7]{May_classifying_spaces}. Similarly, if $X$ is a right $\capO$-module and $Y$ is a left $\capO$-module, then the \emph{simplicial bar construction} (or two-sided simplicial bar construction) $\BAR(X,\capO,Y)$ in $\sSymSeq$ looks like (showing only the face maps)
\begin{align*}
\xymatrix{
  X\circ Y & 
  X\circ\capO\circ Y\ar@<-0.5ex>[l]_-{m\circ\id}\ar@<0.5ex>[l]^-{\id\circ m} & 
  X\circ\capO\circ\capO\circ Y\ar@<-1.0ex>[l]\ar[l]\ar@<1.0ex>[l] &
  \cdots\ar@<1.5ex>[l]\ar@<0.5ex>[l]\ar@<-0.5ex>[l]\ar@<-1.5ex>[l]
}
\end{align*}
and is defined objectwise by $\BAR(X,\capO,Y)_k:=X\circ\capO^{\circ k}\circ Y$ with the obvious face and degeneracy maps induced by the multiplication and unit maps. 
\end{defn}

Sometimes the simplicial bar construction has the additional structure of a simplicial $\capO$-algebra, simplicial left $\capO$-module, or simplicial right $\capO$-module. If $X$ is an $\capO$-algebra, then it is easy to check that there are isomorphisms 
\begin{align*}
  X\Iso\colim^{\AlgO}\Bigl(
  \xymatrix{
    \capO\circ(X) & \capO\circ\capO\circ(X)
    \ar@<-0.5ex>[l]_-{m\circ(\id)}\ar@<0.5ex>[l]^-{\id\circ(m)}
  }
  \Bigr)
  \Iso\colim^{\AlgO}_{\Delta^\op}\BAR(\capO,\capO,X)
\end{align*}
of $\capO$-algebras, natural in $X$. In other words, every $\capO$-algebra  (resp. left $\capO$-module) $X$ is naturally isomorphic to the colimit of its simplicial resolution $\BAR(\capO,\capO,X)$, and Theorem \ref{thm:fattened_replacement} can be understood as a homotopical version of this. The following proposition, which will be used in the proof of Theorem \ref{thm:fattened_replacement}, follows from \cite[9.8]{May} as we indicate below.

\begin{prop}\label{prop:simplicial_resolution_map}
Let $\capO$ be an operad in symmetric spectra or unbounded chain complexes over $\unit$. If $X$ is an $\capO$-algebra (resp. left $\capO$-module), then the natural map $|\BAR(\capO,\capO,X)|\rarrow|X\cdot\Delta[0]|\Iso X$ is a weak equivalence. Here, $\unit$ is any commutative ring.
\end{prop}

\begin{proof}
Let $X\in\LtO$ and consider $\BAR(\capO,\capO,X)$ in the underlying category $\sSymSeq$. The unit map $\function{\eta}{I}{\capO}$ induces maps
\begin{align*}
  \function{s_{-1}:=\eta\circ\id^{\circ k}\circ\id}{\capO^{\circ k}\circ X}{\capO\circ\capO^{\circ k}\circ X},\quad\quad\text{($k\geq 0$)},
\end{align*}
in the underlying category $\SymSeq$ which satisfy the relations
\begin{align*}
  d_0s_{-1} = \id,\quad\quad
  d_is_{-1} = s_{-1}d_{i-1},\quad\quad
  s_{-1}s_j = s_{j+1}s_{-1},
\end{align*}
for all $i>0$ and $j\geq -1$. The maps $s_{-1}$ are sometimes called \emph{extra degeneracy maps} for $\BAR(\capO,\capO,X)$ since these relations are the usual simplicial identities \cite[I.1]{Goerss_Jardine} applied to the maps $s_{-1}$. Consider the maps $s$ and $r$ in $\sSymSeq$ of the form
\begin{align*}
\xymatrix{
  X\cdot\Delta[0]\ar[r]^-{s} & \BAR(\capO,\capO,X)\ar[r]^-{r} & 
  X\cdot\Delta[0]
}
\end{align*}
and induced by $X\xrightarrow{s_{-1}}\capO\circ X$ and $\capO\circ X\xrightarrow{d_0:=m} X$, respectively. It is easy to check that $rs=\id$ and by \cite[9.8]{May} the extra degeneracy maps determine a simplicial homotopy from $sr$ to the identity map; hence the map $|r|$ is a weak equivalence.
\end{proof}

\begin{proof}[Proof of Theorem \ref{thm:fattened_replacement}]
Consider any $X\in\LtO$. For ease of notation purposes, define $B(X):=\BAR(\capO,\capO,X)\in\sLtO$ and $\Delta(X):=X\cdot\Delta[0]\in\sLtO$. By Theorem \ref{MainTheorem3} there is a commutative diagram
\begin{align*}
\xymatrix{
  \hocolim\limits^\LtO_{\Delta^\op} B(X)\ar[d]^{(*)}\ar@{-}[r]^-{\wequiv} &
  |B(X)|\ar[d]^{(**)}\\
  \hocolim\limits^\LtO_{\Delta^\op} \Delta(X)\ar@{-}[r]^-{\wequiv} &
  |\Delta(X)|
}
\end{align*}
with each row a zig-zag of weak equivalences. We know that $(**)$ is a weak equivalence by Proposition \ref{prop:simplicial_resolution_map}, hence $(*)$ is a weak equivalence. The map $\emptyset\rarrow X$ factors functorially $\emptyset\rarrow X^c\rarrow X$ in $\LtO$ as a cofibration followed by an acyclic fibration. This gives natural zig-zags of weak equivalences
\begin{align*}
  \hocolim\limits^\LtO_{\Delta^\op} B(X)\wequiv
  \hocolim\limits^\LtO_{\Delta^\op} \Delta(X)\wequiv
  \hocolim\limits^\LtO_{\Delta^\op} \Delta (X^c)\wequiv
  \colim\limits^\LtO_{\Delta^\op} \Delta (X^c)\Iso 
  X^c\wequiv X
\end{align*}
in $\LtO$, which finishes the proof.
\end{proof}

The following is a special case of Proposition \ref{prop:hocolim_commutes_with_left_quillen_M}.

\begin{prop}\label{prop:commuting_with_hocolim}
Let $\function{f}{\capO}{\capO'}$ be a morphism of operads in symmetric spectra or unbounded chain complexes over $\unit$. If $X$ is a simplicial $\capO$-algebra (resp. simplicial left $\capO$-module), then there is a zig-zag of weak equivalences
\begin{align*}
  \LL f_*\bigl(\hocolim\limits^{\Alg_{\capO}}_{\Delta^\op}X\bigr)
  &\wequiv
  \hocolim\limits^{\Alg_{\capO'}}_{\Delta^\op}\LL f_* (X) \\
  \Bigl(
  \text{resp.}\quad
  \LL f_*\bigl(\hocolim\limits^{\Lt_{\capO}}_{\Delta^\op}X\bigr)
  &\wequiv
  \hocolim\limits^{\Lt_{\capO'}}_{\Delta^\op}\LL f_*(X)
    \Bigr)
\end{align*}
natural in $X$. Here, $\LL f_*$ is the total left derived functor of $f_*$ and $\unit$ is any field of characteristic zero. 
\end{prop}

\begin{proof}[Proof of Theorem \ref{thm:bar_calculates_derived_circle}]
For each $X\in\LtO$, consider the zig-zags of weak equivalences
\begin{align*}
  \LL f_*(X)
  &\wequiv \LL f_* 
  \Bigl(\hocolim\limits^\LtO_{\Delta^\op}\BAR(\capO,\capO,X)\Bigr) 
  \wequiv \hocolim\limits^{\Lt_{\capO'}}_{\Delta^\op} 
  \LL f_* \bigl(\BAR(\capO,\capO,X)\bigr)\\
  &\wequiv \hocolim\limits^{\Lt_{\capO'}}_{\Delta^\op} 
  f_* \bigl(\BAR(\capO,\capO,X)\bigr)
  \wequiv \hocolim^{\Lt_{\capO'}}\limits_{\Delta^\op} \BAR(\capO',\capO,X) \\
  &\wequiv |\BAR(\capO',\capO,X)|
\end{align*}
in the underlying category $\SymSeq$; these weak equivalences follow immediately from Theorem \ref{thm:fattened_replacement}, Proposition \ref{prop:commuting_with_hocolim}, and Theorem \ref{MainTheorem3}. Argue similarly for the case of $\AlgO$.
\end{proof}

\section{Algebras and modules over non-$\Sigma$ operads}
\label{sec:non_sigma_operads}

The purpose of this section is to observe that the main results of this paper have corresponding versions for algebras and modules over non-$\Sigma$ operads. The arguments are the same as in the previous sections, except using the non-$\Sigma$ versions---described in \cite{Harper_Modules}---of the filtrations in Section \ref{sec:filtrations_of_certain_pushouts}.

\begin{defn}
Let $\capO$ be a non-$\Sigma$ operad in symmetric spectra. 
\begin{itemize}
\item[(a)] The \emph{flat stable model structure} on $\AlgO$ (resp. $\LtO$) has weak equivalences the stable equivalences (resp. objectwise stable equivalences) and fibrations the flat stable fibrations (resp. objectwise flat stable fibrations).
\item[(b)] The \emph{stable model structure} on $\AlgO$ (resp. $\LtO$) has weak equivalences the stable equivalences (resp. objectwise stable equivalences) and fibrations the stable fibrations (resp. objectwise stable fibrations).
\end{itemize}
\end{defn}

\begin{defn}
Let $\capO$ be a non-$\Sigma$ operad in unbounded chain complexes over $\unit$. The model structure on $\AlgO$ (resp. $\LtO$) has weak equivalences the homology isomorphisms (resp. objectwise homology isomorphisms) and fibrations the dimensionwise surjections (resp. objectwise dimensionwise surjections). Here, $\unit$ is any commutative ring.
\end{defn}

The above model structures are proved in \cite{Harper_Modules} to exist. The following is a non-$\Sigma$ operad version of Theorem \ref{MainTheorem3}.

\begin{thm}
Let $\capO$ be a non-$\Sigma$ operad in symmetric spectra or unbounded chain complexes over $\unit$. If $X$ is a simplicial $\capO$-algebra (resp. simplicial left $\capO$-module), then there are zig-zags of weak equivalences
\begin{align*}
  U\hocolim\limits^{\AlgO}_{\Delta^\op}X & \wequiv |U X|\wequiv
  \hocolim\limits_{\Delta^\op}U X  \\
  \Bigl(
  \text{resp.}\quad
  U\hocolim\limits^{\LtO}_{\Delta^\op}X & \wequiv |U X|\wequiv
  \hocolim\limits_{\Delta^\op}U X 
  \Bigr)
\end{align*}
natural in $X$. Here, $U$ is the forgetful functor and $\unit$ is any commutative ring.
\end{thm}

The following is a non-$\Sigma$ operad version of Theorem \ref{thm:fattened_replacement}.

\begin{thm}
Let $\capO$ be a non-$\Sigma$ operad in symmetric spectra or unbounded chain complexes over $\unit$. If $X$ is an $\capO$-algebra (resp. left $\capO$-module), then there is a zig-zag of weak equivalences
\begin{align*}
  X & \wequiv \hocolim\limits^{\AlgO}_{\Delta^\op}\BAR(\capO,\capO,X)\\
  \Bigl(
  \text{resp.}\quad
  X &\wequiv \hocolim\limits^{\LtO}_{\Delta^\op}\BAR(\capO,\capO,X)
  \Bigr)
\end{align*}
in $\AlgO$ (resp. $\LtO$), natural in $X$. Here, $\unit$ is any commutative ring.
\end{thm}

The following is a non-$\Sigma$ operad version of Theorem \ref{MainTheorem4}.

\begin{thm}
Let $\capO$ be a non-$\Sigma$ operad in symmetric spectra or unbounded chain complexes over $\unit$. If $\function{f}{\capO}{\capO'}$ is a map of non-$\Sigma$ operads, then the adjunction
\begin{align*}
\xymatrix{
  \Alg_{\capO}\ar@<0.5ex>[r]^-{f_*} & \Alg_{\capO'},\ar@<0.5ex>[l]^-{f^*}
}
\quad\quad
\Bigl(
\text{resp.}\quad
\xymatrix{
  \Lt_{\capO}\ar@<0.5ex>[r]^-{f_*} & \Lt_{\capO'},\ar@<0.5ex>[l]^-{f^*}
}
\Bigr)
\end{align*} 
is a Quillen adjunction with left adjoint on top and $f^*$ the forgetful functor. If furthermore, $f$ is an objectwise weak equivalence, then the adjunction is a Quillen equivalence, and hence induces an equivalence on the homotopy categories. Here, $\unit$ is any commutative ring.
\end{thm}

The following is a non-$\Sigma$ operad version of Theorem \ref{thm:bar_calculates_derived_circle}.

\begin{thm}\label{thm:non_sigma_change_of_operads}
Let $\function{f}{\capO}{\capO'}$ be a morphism of non-$\Sigma$ operads in symmetric spectra or unbounded chain complexes over $\unit$. Let $X$ be an $\capO$-algebra (resp. left $\capO$-module). If the simplicial bar construction $\BAR(\capO,\capO,X)$ is objectwise cofibrant in $\AlgO$ (resp. $\LtO$), then there is a zig-zag of weak equivalences
\begin{align*}
  \LL f_*(X)&\wequiv
  |\BAR(\capO',\capO,X)|
\end{align*}
in the underlying category, natural in such $X$. Here, $\LL f_*$ is the total left derived functor of $f_*$ and $\unit$ is any commutative ring.
\end{thm}

\section{Operads in chain complexes over a commutative ring}
\label{sec:chain_complexes_ring}

The purpose of this section is to observe that the main results of this paper remain true in the context of unbounded chain complexes over a commutative ring, provided that the desired model category structures on $\capO$-algebras and left $\capO$-modules exist; the arguments are the same as in the previous sections. Some approaches to establishing an appropriate homotopy theory in this context are studied in \cite{Berger_Moerdijk, Spitzweck}.

\begin{assumption}\label{BasicAssumption}
From now on in this section, we assume that $\capO$ is an operad in unbounded chain complexes over $\unit$ such that the following model structure exists on $\AlgO$ (resp. $\LtO$): the model structure on $\AlgO$ (resp. $\LtO$) has weak equivalences the homology isomorphisms (resp. objectwise homology isomorphisms) and fibrations the dimensionwise surjections (resp. objectwise dimensionwise surjections). Here, $\unit$ is any commutative ring.
\end{assumption}

\begin{thm}\label{thm:forgetful_commutative_ring_case}
Let $\capO$ be an operad in unbounded chain complexes over $\unit$. Assume that $\capO$ satisfies Basic Assumption \ref{BasicAssumption}. If $X$ is a simplicial $\capO$-algebra (resp. simplicial left $\capO$-module), then there are zig-zags of weak equivalences
\begin{align*}
  U\hocolim\limits^{\AlgO}_{\Delta^\op}X & \wequiv |U X|\wequiv
  \hocolim\limits_{\Delta^\op}U X \\
  \Bigl(
  \text{resp.}\quad
  U\hocolim\limits^{\LtO}_{\Delta^\op}X & \wequiv |U X|\wequiv
  \hocolim\limits_{\Delta^\op}U X 
  \Bigr)
\end{align*}
natural in $X$. Here, $U$ is the forgetful functor and $\unit$ is any commutative ring.
\end{thm}

\begin{thm}
Let $\capO$ be an operad in unbounded chain complexes over $\unit$. Assume that $\capO$ satisfies Basic Assumption \ref{BasicAssumption}. If $X$ is an $\capO$-algebra (resp. left $\capO$-module), then there is a zig-zag of weak equivalences
\begin{align*}
  X & \wequiv \hocolim\limits^{\AlgO}_{\Delta^\op}\BAR(\capO,\capO,X)\\
  \Bigl(
  \text{resp.}\quad
  X &\wequiv \hocolim\limits^{\LtO}_{\Delta^\op}\BAR(\capO,\capO,X)
  \Bigr)
\end{align*}
in $\AlgO$ (resp. $\LtO$), natural in $X$. Here, $\unit$ is any commutative ring.
\end{thm}

\begin{thm}
Let $\function{f}{\capO}{\capO'}$ be a map of operads in unbounded chain complexes over $\unit$. Assume that $\capO$ and $\capO'$ satisfy Basic Assumption \ref{BasicAssumption}. Then the adjunction
\begin{align*}
\xymatrix{
  \Alg_{\capO}\ar@<0.5ex>[r]^-{f_*} & \Alg_{\capO'}\ar@<0.5ex>[l]^-{f^*}
}
\quad\quad
\Bigl(
\text{resp.}\quad
\xymatrix{
  \Lt_{\capO}\ar@<0.5ex>[r]^-{f_*} & \Lt_{\capO'}\ar@<0.5ex>[l]^-{f^*}
}
\Bigr)
\end{align*} 
is a Quillen adjunction with left adjoint on top and $f^*$ the forgetful functor. If furthermore, $f$ is an objectwise weak equivalence and both $\capO$ and $\capO'$ are cofibrant in the underlying category $\SymSeq$, then the adjunction is a Quillen equivalence, and hence induces an equivalence on the homotopy categories. Here, $\unit$ is any commutative ring.
\end{thm}

\begin{thm}\label{thm:sigma_change_of_operads}
Let $\function{f}{\capO}{\capO'}$ be a morphism of operads in unbounded chain complexes over $\unit$. Assume that $\capO$ and $\capO'$ satisfy Basic Assumption \ref{BasicAssumption}. Let $X$ be an $\capO$-algebra (resp. left $\capO$-module). If the simplicial bar construction $\BAR(\capO,\capO,X)$ is objectwise cofibrant in $\AlgO$ (resp. $\LtO$), then there is a zig-zag of weak equivalences
\begin{align*}
  \LL f_*(X)&\wequiv
  |\BAR(\capO',\capO,X)|
\end{align*}
in the underlying category, natural in such $X$. Here, $\LL f_*$ is the total left derived functor of $f_*$ and $\unit$ is any commutative ring.
\end{thm}

\section{Right modules over operads}
\label{sec:right_modules}

The purpose of this section is to observe that several of the results of this paper have corresponding versions for right modules over operads; the corresponding arguments are substantially less complicated, since colimits in right modules over an operad are calculated in the underlying category of symmetric sequences \cite{Harper_Modules}.

\begin{defn}
\label{defn:right_modules_spectra}
Let $\capO$ be an operad in symmetric spectra. 
\begin{itemize}
\item[(a)] The \emph{flat stable model structure} on $\RtO$ has weak equivalences the objectwise stable equivalences and fibrations the objectwise flat stable fibrations.
\item[(b)] The \emph{stable model structure} on $\RtO$ has weak equivalences the objectwise stable equivalences and fibrations the objectwise stable fibrations.
\end{itemize}
\end{defn}

\begin{defn}
\label{defn:right_modules_chain_complexes}
Let $\capO$ be an operad in unbounded chain complexes over $\unit$. The model structure on $\RtO$ has weak equivalences the objectwise homology isomorphisms and fibrations the objectwise dimensionwise surjections. Here, $\unit$ is any commutative ring.
\end{defn}

The existence of the model structures in Definition \ref{defn:right_modules_spectra} follows easily from the corresponding argument in \cite{Harper_Spectra} together with the following properties: for symmetric spectra with the flat stable model structure, smashing with a cofibrant symmetric spectrum preserves weak equivalences, and the generating (acyclic) cofibrations have cofibrant domains. Similarly, the existence of the model structure in Definition \ref{defn:right_modules_chain_complexes} follows easily from the following properties: for unbounded chain complexes over $\unit$, tensoring with a cofibrant chain complex preserves weak equivalences, and the generating (acyclic) cofibrations have cofibrant domains. Similar model structures are considered in \cite{Fresse_modules}.

The following is a right $\capO$-module version of Theorem \ref{MainTheorem3}.

\begin{thm}
Let $\capO$ be an operad in symmetric spectra or unbounded chain complexes over $\unit$. If $X$ is a simplicial right $\capO$-module, then there are zig-zags of weak equivalences
\begin{align*}
  U\hocolim\limits^{\RtO}_{\Delta^\op}X & \wequiv |U X| \wequiv
  \hocolim\limits_{\Delta^\op}U X 
\end{align*}
natural in $X$. Here, $U$ is the forgetful functor and $\unit$ is any commutative ring.
\end{thm}

\begin{proof}
Argue as in the proof of Theorem \ref{MainTheorem3} and Proposition \ref{prop:comparing_hocolim_and_realzn}, except replace \eqref{eq:simplicial_glueing_chain_complex} with pushout diagrams of the form
\begin{align*}
\xymatrix{
  \Delta[z]\cdot X\circ\capO\ar[r]\ar[d] & Z_0\ar[d]\\
  \Delta[z]\cdot Y\circ\capO\ar[r] & Z_1
}
\end{align*}
in $\sRtO$, with $X\rarrow Y$ a generating cofibration in $\SymSeq$, and note that pushouts in $\sRtO$ are calculated in the underlying category $\sSymSeq$.
\end{proof}

The following is a right $\capO$-module version of Theorem \ref{thm:fattened_replacement}.

\begin{thm}
Let $\capO$ be an operad in symmetric spectra or unbounded chain complexes over $\unit$. If $X$ is a right $\capO$-module, then there is a zig-zag of weak equivalences
\begin{align*}
  X &\wequiv \hocolim\limits^{\RtO}_{\Delta^\op}\BAR(X,\capO,\capO)
\end{align*}
in $\RtO$, natural in $X$. Here, $\unit$ is any commutative ring.
\end{thm}

\section{Proofs}
\label{sec:proofs}

The purpose of this section is to prove Propositions \ref{prop:realzn_homotopy_meaningful}, \ref{prop:realzn_preserves_monomorphisms}, and \ref{prop:tot_of_normalization}. A first step is to recall the decomposition of simplicial chain complexes which lies at the heart of the Dold-Kan correspondence (Section \ref{sec:decompositions_dold_kan}). In Section \ref{sec:skeletal_filtrations} we describe the skeletal filtration of realization, which is a key ingredient in the homotopical analysis of the realization functors (Section \ref{sec:homotopical_analysis_proofs_end}). 

\begin{assumption}
From now on in this section, we assume that $\unit$ is any commutative ring.
\end{assumption}

\subsection{Decomposition of simplicial chain complexes}
\label{sec:decompositions_dold_kan}

The purpose of this section is to recall the decomposition of simplicial chain complexes described in Proposition \ref{prop:natural_decomposition}.

\begin{defn}
\label{defn:normalization_functor}
Let $X$ be a simplicial unbounded chain complex over $\unit$ (resp. simplicial $\unit$-module) and $n\geq 0$. Define the subobject $\NN X_n\subset X_n$ by 
\begin{align*}
  \NN X_0 := X_0,\quad\quad
  &\NN X_n := \bigcap\limits_{0\leq i \leq n-1}\ker(d_i)\subsetof X_n,
  \quad\quad
  (n\geq 1).
\end{align*}
\end{defn}

\begin{prop}\label{prop:natural_decomposition}
Let $X$ be a simplicial unbounded chain complex over $\unit$ (resp. simplicial $\unit$-module). There is a natural isomorphism $\Psi$ in $\sChaincx_\unit$ (resp. $\sMod_\unit$) defined objectwise by
\begin{align}
\label{eq:natural_decomposition}
  \Psi_n\colon\thinspace \coprod\limits_{
  \substack{[n]\twoheadrightarrow [k]\\ 
  \text{in $\Delta$}}}\NN X_k\xrightarrow{\Iso} X_n.
\end{align}
Here, the coproduct is indexed over the set of all surjections in $\Delta$ of the form $\function{\xi}{[n]}{[k]}$, and $\Psi_n$ is the natural map induced by the corresponding maps
$
  \NN X_k\subset X_k\xrightarrow{\xi^*} X_n.
$
\end{prop}

In other words, each $X$ in $\sChaincx_\unit$ (resp. $\sMod_\unit$) is naturally isomorphic to a simplicial object of the form (showing only the face maps)
\begin{align*}
\xymatrix{
  \NN X_0 & 
  \NN X_0\amalg\NN X_1\ar@<-0.5ex>[l]\ar@<0.5ex>[l] &
  \NN X_0\amalg\NN X_1\amalg\NN X_1\amalg\NN X_2
  \ar@<-1.0ex>[l]\ar[l]\ar@<1.0ex>[l] &
  \dotsb\ar@<-1.5ex>[l]\ar@<-0.5ex>[l]\ar@<0.5ex>[l]\ar@<1.5ex>[l]
}
\end{align*}

\begin{proof}[Proof of Proposition \ref{prop:natural_decomposition}]
This follows from the Dold-Kan correspondence \cite[III.2]{Goerss_Jardine}, \cite[8.4]{Weibel} that normalization $\NN$ fits into the following 
\begin{align}
\label{eq:Dold_Kan_correspondence_adjunction}
\xymatrix{
  \sChaincx_\unit\ar@<0.5ex>[r]^-{\NN} & 
  \Chaincx^+(\Chaincx_\unit)\ar@<0.5ex>[l]\quad\quad
  \Bigl(
  \text{resp.}\quad
  \sMod_\unit\ar@<0.5ex>[r]^-{\NN} & 
  \Chaincx^+(\Mod_\unit)\ar@<0.5ex>[l]
  \Bigr)
}
\end{align}
equivalence of categories. Here, $\Chaincx^+(\Chaincx_\unit)$ (resp. $\Chaincx^+(\Mod_\unit)$) denotes the category of non-negative chain complexes in $\Chaincx_\unit$ (resp. in $\Mod_\unit$).
\end{proof}

\subsection{Skeletal filtration of realization}
\label{sec:skeletal_filtrations}

The purpose of this section is to describe the skeletal filtration of realization given in Proposition \ref{prop:skeletal_glueing}.

\begin{defn} 
Let $n\geq 0$. The functors $\realzn_n$ for simplicial symmetric spectra and simplicial unbounded chain complexes over $\unit$ are defined objectwise by the coends
\begin{align*}
  \functor{\realzn_n}{\sSpectra}{\Spectra},
  &\quad\quad
  X\longmapsto X\Smash_{\Delta}\Sk_n\Delta[-]_+\ , \\
  \functor{\realzn_n}{\sChaincx_\unit}{\Chaincx_\unit},
  &\quad\quad
  X\longmapsto X\tensor_{\Delta}\NN\unit\Sk_n\Delta[-].
\end{align*}
\end{defn}

\begin{prop}
Let $n\geq 0$. The functors $\realzn_n$ fit into adjunctions
\begin{align*}
\xymatrix{
  \sSpectra
  \ar@<0.5ex>[r]^-{\realzn_n} & \Spectra,\ar@<0.5ex>[l]
}\quad\quad
\xymatrix{
  \sChaincx_\unit
  \ar@<0.5ex>[r]^-{\realzn_n} & \Chaincx_\unit,\ar@<0.5ex>[l]
}
\end{align*}
with left adjoints on top. Each adjunction is a Quillen pair. 
\end{prop}

\begin{proof}
This follows as in the proof of Proposition \ref{prop:realzns_fit_into_adjunctions}.
\end{proof}

\begin{prop}
Let $X$ be a symmetric spectrum (resp. unbounded chain complex over $\unit$) and $n\geq 0$. There is a natural isomorphism $\realzn_n(X\cdot\Delta[0])\Iso X$. 
\end{prop}

\begin{proof}
This follows from uniqueness of left adjoints (up to isomorphism).
\end{proof}

\begin{prop}\label{prop:skeletal_filtration_useful}
Let $X$ be a simplicial symmetric spectrum (resp. simplicial unbounded chain complex over $\unit$). The realization $|X|$ is naturally isomorphic to a filtered colimit of the form
\begin{align*}
\xymatrix{
  |X|\Iso
  \colim\Bigl(
  \realzn_0(X)\ar[r] & \realzn_1(X)\ar[r] & \realzn_2(X)\ar[r] & \cdots
  \Bigr).
}
\end{align*}
\end{prop}

\begin{proof}
Consider the case of simplicial unbounded chain complexes over $\unit$. We know that $\Delta[-]\Iso\colim_n\Sk_n\Delta[-]$ in $\sSet^\Delta$. Since the functors 
$\functor{\NN\unit}{\sSet^\Delta}{\Chaincx_\unit^\Delta}$ and
$\functor{X\tensor_\Delta-}{\Chaincx_\unit^\Delta}{\Chaincx_\unit}$ preserve colimiting cones, it follows that there are natural isomorphisms
\begin{align*}
  \NN\unit\Delta[-]&
  \Iso\colim_n\NN\unit\Sk_n\Delta[-],\\
  X\tensor_\Delta\NN\unit\Delta[-]
  &\Iso
  \colim_n X\tensor_\Delta\NN\unit\Sk_n\Delta[-].
\end{align*}
Consider the case of simplicial symmetric spectra. We know there is an  isomorphism $\Delta[-]_+\Iso\colim_n\Sk_n\Delta[-]_+$ in $\sSet_*^\Delta$. Since the functors
\begin{align*}
  \functor{S\tensor G_0}{\sSet_*^\Delta}{\bigl(\Spectra\bigr)^\Delta},
  \quad\quad
  \functor{X\Smash_\Delta-}{\bigl(\Spectra\bigr)^\Delta}{\Spectra},
\end{align*}
preserve colimiting cones, a similar argument finishes the proof.
\end{proof}

\begin{defn}
Let $X$ be a simplicial symmetric spectrum (resp. simplicial unbounded chain complex over $\unit$) and $n\geq 0$. Define the subobject $DX_n\subset X_n$ by 
\begin{align*}
  DX_0 := *,\quad\quad
  &DX_n := \bigcup\limits_{0\leq i \leq n-1}
  s_iX_{n-1}\subsetof X_n,
  \quad\quad
  (n\geq 1)\\
  \Bigl(
  \text{resp.}\quad
  DX_0 := *,\quad\quad
  &DX_n := \sum\limits_{0\leq i \leq n-1}
  s_iX_{n-1}\subsetof X_n,
  \quad\quad
  (n\geq 1)
  \Bigr).
\end{align*}
We refer to $DX_n$ as the \emph{degenerate subobject} of $X_n$.
\end{defn}

\begin{prop}\label{prop:unions_and_intersections}
Let $X$ be a simplicial symmetric spectrum (resp. simplicial unbounded chain complex over $\unit$) and $n\geq 1$. There are pushout diagrams
\begin{align}
\label{eq:unions_and_intersections_spectra}
&\xymatrix{
  DX_{n}\Smash\partial\Delta[n]_+\ar[d]\ar[r] & 
  X_{n}\Smash\partial\Delta[n]_+\ar[d]\\
  DX_{n}\Smash\Delta[n]_+\ar[r] & 
  (DX_{n}\Smash\Delta[n]_+)\cup
  (X_{n}\Smash\partial\Delta[n]_+)
}\\
\label{eq:unions_and_intersections}
\text{resp.}\quad
&\xymatrix{
  DX_{n}\tensor\NN\unit\partial\Delta[n]\ar[d]\ar[r] & 
  X_{n}\tensor\NN\unit\partial\Delta[n]\ar[d]\\
  DX_{n}\tensor\NN\unit\Delta[n]\ar[r] & 
  (DX_{n}\tensor\NN\unit\Delta[n])+
  (X_{n}\tensor\NN\unit\partial\Delta[n])
}
\end{align}
in $\Spectra$ (resp. in $\Chaincx_\unit$). The maps in \eqref{eq:unions_and_intersections_spectra} and \eqref{eq:unions_and_intersections} are monomorphisms. 
\end{prop}

\begin{proof}
In the case of simplicial symmetric spectra, the pushout diagrams \eqref{eq:unions_and_intersections_spectra} follow from the corresponding pushout diagrams for a bisimplicial set \cite[IV.1]{Goerss_Jardine}. In the case of simplicial unbounded chain complexes over $\unit$, use Proposition \ref{prop:natural_decomposition} to reduce to verifying that the diagram
\begin{align*}
\xymatrix{
  DX_n\tensor\NN\unit\partial\Delta[n]\ar@{=}[r]\ar[d] & 
  DX_n\tensor\NN\unit\partial\Delta[n]\ar[d]\\
  DX_n\tensor\NN\unit\Delta[n]\ar@{=}[r] &
  DX_n\tensor\NN\unit\Delta[n]
}
\end{align*}
is a pushout diagram.
\end{proof}

\begin{prop}\label{prop:skeletal_glueing}
Let $X$ be a simplicial symmetric spectrum (resp. simplicial unbounded chain complex over $\unit$) and $n\geq 1$. There are pushout diagrams
\begin{align}
\label{eq:skeletal_glueing_spectra}
&\xymatrix{
  (DX_{n}\Smash\Delta[n]_+)\cup
  (X_{n}\Smash\partial\Delta[n]_+)\ar[d]\ar[r] & 
  \realzn_{n-1}(X)\ar[d]\\
  X_{n}\Smash\Delta[n]_+\ar[r] & \realzn_{n}(X)
}\\
\label{eq:skeletal_glueing}
\text{resp.}\quad
&\xymatrix{
  (DX_{n}\tensor\NN\unit\Delta[n])+
  (X_{n}\tensor\NN\unit\partial\Delta[n])\ar[d]\ar[r] & 
  \realzn_{n-1}(X)\ar[d]\\
  X_{n}\tensor\NN\unit\Delta[n]\ar[r] & \realzn_{n}(X)
}
\end{align}
in $\Spectra$ (resp. $\Chaincx_\unit$). The vertical maps in \eqref{eq:skeletal_glueing_spectra} and \eqref{eq:skeletal_glueing} are monomorphisms. 
\end{prop}

\begin{proof}
In the case of simplicial symmetric spectra, the pushout diagrams \eqref{eq:skeletal_glueing_spectra} follow from the corresponding pushout diagrams for a bisimplicial set \cite[IV.1]{Goerss_Jardine}. In the case of simplicial unbounded chain complexes over $\unit$, use Proposition \ref{prop:natural_decomposition} to reduce to verifying that the diagram
\begin{align}
\label{eq:glueing_on_non_degenerate_stuff}
\xymatrix{
  \NN X_n\tensor\NN\unit\partial\Delta[n]\ar[r]\ar[d] &
  \realzn_{n-1}(X)\ar[d]\\
  \NN X_n\tensor\NN\unit\Delta[n]\ar[r] &
  \realzn_n(X)
}
\end{align}
is a pushout diagram in $\Chaincx_\unit$, which follows from the simplicial identities and the property that $\functor{\NN\unit}{\sSet}{\Chaincx_\unit}$ preserves colimiting cones.
\end{proof}

\subsection{Proofs}
\label{sec:homotopical_analysis_proofs_end}

\begin{proof}[Proof of Proposition \ref{prop:realzn_preserves_monomorphisms}]
In the case of simplicial symmetric spectra, this follows from the corresponding property for realization of a bisimplicial set \cite[IV.1]{Goerss_Jardine}. Consider the case of simplicial unbounded chain complexes over $\unit$. Use Proposition \ref{prop:natural_decomposition} to argue that $\functor{\NN}{\sChaincx_\unit}{\Chaincx_\unit}$ preserves monomorphisms; either use the Dold-Kan correspondence \eqref{eq:Dold_Kan_correspondence_adjunction} and note that right adjoints preserve monomorphisms, or use \eqref{eq:natural_decomposition} and note that monomorphisms are preserved under retracts. To finish the argument, forget differentials and use the pushout diagrams \eqref{eq:glueing_on_non_degenerate_stuff} to give a particularly simple filtration of $\function{|f|}{|X|}{|Y|}$ in the underlying category of graded $\unit$-modules. Since $\NN X_n\rarrow \NN Y_n$ is a monomorphism for each $n\geq 0$, it follows from this filtration that $|f|$ is a monomorphism.
\end{proof}

\begin{prop}\label{prop:degenerate_subobject}
If $\function{f}{X}{Y}$ in $\sSpectra$ (resp. in $\sChaincx_\unit$) is an objectwise weak equivalence, then $\function{Df_n}{DX_n}{DY_n}$ is a weak equivalence for each $n\geq 1$. 
\end{prop}

Before proving this, it will be useful to filter the degenerate subobjects.

\begin{defn}
Let $X$ be a simplicial symmetric spectrum (resp. simplicial unbounded chain complex over $\unit$) and $n\geq 1$. For each $0\leq r\leq n-1$, define the subobjects $s_{[r]}X_{n-1}\subsetof X_n$ by
\begin{align*}
   &s_{[r]}X_{n-1} := \bigcup\limits_{0\leq i \leq r}
   s_iX_{n-1}\subsetof X_n\quad\quad
   \Bigl(
   \text{resp.}\quad
   &s_{[r]}X_{n-1} := \sum\limits_{0\leq i \leq r}
   s_iX_{n-1}\subsetof X_n
   \Bigr).
\end{align*}
In particular, $s_{[0]}X_{n-1}\Iso X_{n-1}$ and $s_{[n-1]}X_{n-1}=DX_n$.
\end{defn}

\begin{prop}\label{prop:degenerate_subobject_filtration}
Let $X$ be a simplicial symmetric spectrum (resp. simplicial unbounded chain complex over $\unit$) and $n\geq 1$. For each $0\leq r\leq n-1$, the diagram
\begin{align}
\label{eq:filtering_the_degenerate_subobject}
\xymatrix{
  s_{[r]}X_{n-1}\ar[d]^{\subset}\ar[r]^-{s_{r+1}} & 
  s_{[r]}X_n\ar[d]^{\subset}\\
  X_n\ar[r]^-{s_{r+1}} & s_{[r+1]}X_n
}
\end{align}
is a pushout diagram. The maps in \eqref{eq:filtering_the_degenerate_subobject} are monomorphisms. 
\end{prop}

\begin{proof}
Consider the case of simplicial symmetric spectra. This follows from the corresponding pushout diagrams for a bisimplicial set \cite[IV.1]{Goerss_Jardine}. Consider the case of simplicial unbounded chain complexes over $\unit$. This follows from Proposition \ref{prop:natural_decomposition} and the simplicial identities.
\end{proof}

\begin{proof}[Proof of Proposition \ref{prop:degenerate_subobject}]
Consider the case of simplicial symmetric spectra (resp. simplicial unbounded chain complexes over $\unit$). We know that $Df_1$ is a weak equivalence since $Df_1\Iso f_0$. Let $n=2$. By Proposition \ref{prop:degenerate_subobject_filtration}, $Df_2$ fits into the commutative diagram
\begin{align*}
\xymatrix{
  s_0X_0\ar[d]\ar[r]^-{s_1}\ar[drr]^(0.7){(a)} & 
  s_0X_1\ar[d]\ar[drr]^(0.7){(d)}\\
  X_1\ar[r]^-{s_1}\ar[d]\ar[drr]^(0.7){(b)} & 
  DX_2\ar[d]\ar[drr]^(0.7){Df_2} & 
  s_0Y_0\ar[r]_-{s_1}\ar[d] &
  s_0Y_1\ar[d]\\
  X_1/s_0X_0\ar[r]^-{\iso}\ar[drr]^(0.7){(c)} & 
  DX_2/s_0X_1\ar[drr]^(0.7){(c)} &
  Y_1\ar[r]_-{s_1}\ar[d] &
  DY_2\ar[d] \\
   & & Y_1/s_0Y_0\ar[r]_-{\iso} & DY_2/s_0Y_1
}
\end{align*}
Since we know the maps $(a)$ and $(b)$ are weak equivalences, it follows that each map $(c)$ is a weak equivalence. Since we know the map $(d)$ is a weak equivalence, it follows that $Df_2$ is a weak equivalence. Similarly, use Proposition \ref{prop:degenerate_subobject_filtration} in an induction argument to verify that $\function{Df_n}{DX_n}{DY_n}$ is a weak equivalence for each $n\geq 3$.
\end{proof}

\begin{proof}[Proof of Proposition \ref{prop:realzn_homotopy_meaningful}]
Consider the case of simplicial symmetric spectra (resp. simplicial unbounded chain complexes over $\unit$). Skeletal filtration gives a commutative diagram of the form
\begin{align*}
\xymatrix{
  \realzn_0(X)\ar[d]^{\realzn_0(f)}\ar[r] & 
  \realzn_1(X)\ar[d]^{\realzn_1(f)}\ar[r] & 
  \realzn_2(X)\ar[d]^{\realzn_2(f)}\ar[r] & 
  \dotsb\\
  \realzn_0(Y)\ar[r] & 
  \realzn_1(Y)\ar[r] & 
  \realzn_2(Y)\ar[r] & 
  \dotsb
}
\end{align*}
We know that $\realzn_0(f)\Iso f_0$ is a weak equivalence. Since the horizontal maps are monomorphisms and we know that
\begin{align*}
  \realzn_{n}(X)/\realzn_{n-1}(X)
  &\Iso (X_{n}/D_{n}X)\Smash(\Delta[n]/\partial\Delta[n]) \\
  \Bigl(
  \text{resp.}\quad
  \realzn_{n}(X)/\realzn_{n-1}(X)
  &\Iso (X_{n}/D_{n}X)\tensor (\NN\unit\Delta[n]/\NN\unit\partial\Delta[n])
  \Bigr)
\end{align*}
it is enough to verify that $\function{Df_n}{DX_n}{DY_n}$ is a weak equivalence for each $n\geq 1$, and Proposition \ref{prop:degenerate_subobject} finishes the proof.
\end{proof}

\begin{proof}[Proof of Proposition \ref{prop:tot_of_normalization}]
If $A,B\in\Chaincx_\unit$, define the objectwise tensor $A\tensordot B\in\Chaincx(\Chaincx_\unit)$ such that $A\tensor B=\Tot^\oplus(A\tensordot B)$. It follows that there are natural isomorphisms
\begin{align*}
  X\tensor_\Delta\NN\unit\Delta[-]\Iso
  \Tot^\oplus\bigl(X\tensordot_\Delta\NN\unit\Delta[-]\bigr)\Iso
  \Tot^\oplus\bigl(\NN\unit\Delta[-]\tensordot_\Delta X\bigr).
\end{align*} 
Arguing as in the proof of Proposition \ref{prop:skeletal_filtration_useful}, verify that
\begin{align*}
  \NN\unit\Delta[-]\tensordot_\Delta X \Iso
  \colim_n\bigl(\NN\unit\Sk_n\Delta[-]\tensordot_\Delta X\bigr)
\end{align*}
and use the pushout diagrams
\begin{align*}
\xymatrix{
\NN\unit\partial\Delta[n]\tensordot\NN X_n\ar[d]\ar[r] & 
\NN\unit\Sk_{n-1}\Delta[-]\tensordot_\Delta X\ar[d]\\
\NN\unit\Delta[n]\tensordot\NN X_n\ar[r] &
\NN\unit\Sk_{n}\Delta[-]\tensordot_\Delta X
}
\end{align*}
in $\Chaincx(\Chaincx_\unit)$ to verify that $\NN\unit\Delta[-]\tensordot_\Delta X \Iso \NN(X)$, which finishes the proof.
\end{proof}

\section{Forgetful functors preserve cofibrant objects}
\label{sec:cofibrant_operads}

The purpose of this section is to prove Proposition \ref{prop:cofibrant_operads}, which shows that certain forgetful functors preserve cofibrant objects.

\begin{defn}
\label{defn:cofibrant_operads}
Let $\capO$ be an operad in symmetric spectra (resp. unbounded chain complexes over $\unit$) and consider the underlying category $\SymSeq$ with any of the monoidal model category structures in Section \ref{sec:model_structures_symmetric_spectra}. Then $\capO$ is a  \emph{cofibrant operad} if the following lifting property is satisfied: given a solid diagram 
\begin{align*}
\xymatrix{
  & B\ar[d]^{p}  \\
  \capO\ar@{.>}[ur]^-{\xi}\ar[r] & C
}
\end{align*}
of operad maps such that $p$ is an acyclic fibration in the underlying category $\SymSeq$, then there exists a morphism of operads $\xi$ which makes the diagram commute.
\end{defn}

The following proposition is motivated by a similar argument given in \cite[Section 13.2]{Rezk_notes} and \cite[4.1]{Rezk} for the case of algebras over an operad.

\begin{prop}
\label{prop:cofibrant_operads}
Let $\capO$ be an operad in symmetric spectra (resp. unbounded chain complexes over $\unit$). Consider $\LtO$, $\SymSeq$, $\AlgO$, and $\Spectra$ (resp. $\Chaincx_\unit$) with the same type of model structure (e.g., the positive flat stable model structure). If $\capO$ is a cofibrant operad, then the forgetful functors 
\begin{align*}
  \LtO\rarrow\SymSeq,
\quad\quad
  \AlgO\rarrow\Spectra,
\quad\quad
\Bigl(
\text{resp.}\quad
  \AlgO\rarrow\Chaincx_\unit,
\Bigr)
\end{align*}
preserve cofibrant objects. 
\end{prop}

\begin{proof}
Consider the case of left $\capO$-modules. Let $Y$ be a cofibrant left $\capO$-module and consider the map $*\rarrow Y$ in the underlying category $\SymSeq$. Use functorial factorization in $\SymSeq$ to obtain a diagram
\begin{align*}
\xymatrix{
  {*}\ar[r] & X\ar[r]^-{p} & Y
}
\end{align*}
giving a cofibration followed by an acyclic fibration. We want to show there exists a left $\capO$-module structure on $X$ such that $p$ is a morphism of left $\capO$-modules. Consider the solid diagram
\begin{align*}
\xymatrix{
  & \Map^\circ(X,X)\times_{\Map^\circ(X,Y)}\Map^\circ(Y,Y)
  \ar[d]^{(*)}\ar[r]^-{(**)}
  & \Map^\circ(X,X)\ar[d]^{(\id,p)}\\
  \capO\ar[r]^-{m}\ar@{.>}[ur]^-{\ol{m}} & 
  \Map^\circ(Y,Y)\ar[r]^{(p,\id)} & \Map^\circ(X,Y)
}
\end{align*}
in $\SymSeq$ such that the square is a pull-back diagram. It is easy to verify  there exists an operad structure on $\Map^\circ(X,X)\times_{\Map^\circ(X,Y)}\Map^\circ(Y,Y)$ such that $(*)$ and $(**)$ are morphisms of operads. Since $X$ is cofibrant in $\SymSeq$ and $p$ is an acyclic fibration in $\SymSeq$, we know from \cite{Harper_Modules} that $(\id,p)$ is an acyclic fibration, and hence $(*)$ is an acyclic fibration in $\SymSeq$. By assumption, $\capO$ is a cofibrant operad, hence there exists a morphism of operads $\ol{m}$ which makes the diagram commute. It follows that the composition 
\begin{align*}
  \capO\xrightarrow{\ol{m}}\Map^\circ(X,X)\times_{\Map^\circ(X,Y)}\Map^\circ(Y,Y)
  \xrightarrow{(**)}\Map^\circ(X,X)
\end{align*}
of operad maps determines a left $\capO$-module structure on $X$ such that $p$ is a morphism of left $\capO$-modules. To finish the proof, we want to show that $Y$ is cofibrant in the underlying category $\SymSeq$. Consider the solid commutative diagram
\begin{align*}
\xymatrix{
  \emptyset\ar[d]\ar[r]& X\ar[d]^{p}\\
  Y\ar@{.>}[ur]^-{\xi}\ar@{=}[r] & Y
}
\end{align*}
in $\LtO$. Since $Y$ is cofibrant in $\LtO$ and $p$ is an acyclic fibration, this diagram has a lift $\xi$ in $\LtO$. In particular, $Y$ is a retract of $X$ in the underlying category $\SymSeq$, and noting that $X$ is cofibrant in $\SymSeq$ finishes the proof. Argue similarly for the case of $\capO$-algebras.
\end{proof}

\bibliographystyle{plain}
\bibliography{QuillenHomology.bib}

\end{document}